\documentclass[amscd,amssymb,verbatim,11pt]{amsart}
\numberwithin{equation}{section}
\usepackage{epsfig}
\usepackage{multido}
\usepackage{color}
\usepackage{pstricks}
\usepackage{pst-all}
\usepackage{pst-math}
\usepackage{pst-func}
\usepackage{pst-3dplot}
\usepackage{graphicx}
\usepackage{amsmath}
\usepackage{a4,amssymb}

\newtheorem{thm}{Theorem}[section]
\newtheorem{cor}[thm]{Corollary}
\newtheorem{lem}[thm]{Lemma}

\theoremstyle{definition}

\newtheorem{rem}[thm]{Remark}

\newif\ifShowLabels
\ShowLabelstrue
\newdimen\theight
\def\TeXref#1{
     \leavevmode\vadjust{\setbox0=\hbox{{\tt
            \quad\quad  {\small  \bf #1}}}%
     \theight=\ht0
     \advance\theight  by  \dp0
     \advance\theight  by  \lineskip
     \kern -\theight \vbox  to
     \theight{\rightline{\rlap{\box0}}%
      \vss}%
      }}%

    {\begin{thm}\label{#1} \ifShowLabels \TeXref{#1} \fi}%
    {\end{thm}}

    {\begin{def}\label{#1} \ifShowLabels \TeXref{#1} \fi}%
    {\end{def}}

    {\begin{lem}\label{#1} \ifShowLabels \TeXref{#1} \fi}%
    {\end{lem}}

    {\begin{cor}\label{#1} \ifShowLabels \TeXref{#1} \fi}%
    {\end{cor}}

\newcommand{\eqRef}[1]%
     {\ifShowLabels \TeXref{#1} \fi
      \begin{equation}\label{#1} }

\ShowLabelsfalse

\newcommand{\vsp}{\vskip 1em}

\newcommand{\NI}{\noindent}
\newcommand{\bea}{\begin{eqnarray}}
\newcommand{\eea}{\end{eqnarray}}
\newcommand{\IR}{I\!\!R}
\newcommand{\bas}{\begin{align*}}
\newcommand{\eas}{\end{align*}}
\newcommand{\ba}{\begin{align}}
\newcommand{\ea}{\end{align}}

\newcommand{\be}{\begin{equation}}
\newcommand{\ee}{\end{equation}}
\newcommand{\ben}{\begin{eqnarray*}}
\newcommand{\een}{\end{eqnarray*}}
\newcommand{\lam}{\lambda}
\newcommand{\Om}{\Omega}
\newcommand{\om}{\omega}
\newcommand{\tht}{\theta}
\newcommand{\p}{\partial}
\newcommand{\al}{\alpha}

\newcommand{\g}{\gamma}
\newcommand{\ve}{\varepsilon}
\newcommand{\dl}{\delta}
\newcommand{\D}{\Delta}

\newcommand{\G}{\Gamma}

\newcommand{\s}{\sigma}

\title[ Degenerate parabolic]{On the viscosity solutions to a class of nonlinear degenerate parabolic differential equations}

\author[T. Bhattacharya and L. Marazzi]{Tilak Bhattacharya and Leonardo Marazzi}

\thanks{AMS Math Subject Classification 2010: 35K65, 35K55}

\begin{document}

\centerline{\red To appear in Revista Matem\'atica Complutense} 

\maketitle
\begin{abstract} In this work, we show existence and uniqueness of positive solutions of
{  $H(Du, D^2u)+\chi(t)|Du|^\G-f(u)u_t=0$} in $\Om\times(0, T)$ and $u=h$ on its parabolic boundary.
The operator $H$ satisfies certain homogeneity conditions, $\G>0$ and depends on the degree of homogeneity of $H$, $f>0$, increasing and meets a concavity condition. We also consider the case $f\equiv 1$ and prove existence of solutions without sign restrictions.
\end{abstract}

\section{Introduction and statements of the main results}

In this work, we address the issue of existence and {  uniqueness} of viscosity solutions to a class of nonlinear degenerate
parabolic differential equations {  that are doubly nonlinear. Our main goal is to present a unified approach to studying as diverse a group of equations as possible and could be viewed as a natural outgrowth of the previous
works in \cite{BL2, BL3}. As a result, the current work includes as special instances many of the results proven in these works. }

{  We now describe the class of equations of interest to us.} Let $\Om\subset \IR^n,\;n\ge 2$, be a bounded domain and $T>0$. Let $\p\Om$ denote its boundary and $\overline{\Om}$ its closure. Call $\Om_T=\Om\times (0,T)$ and $P_T$ its parabolic boundary. 

We address existence {  results} and comparison principles for viscosity solutions to
\bea\label{sec2.001}
&&H(Du, D^2u)+\chi(t)|Du|^\G-f(u)u_t=0,\;\mbox{in $\Om_T$,}\nonumber\\
&&u(x,0)=i(x),\;\forall x\in \Om\;\;\mbox{and}\;\;u(x,t)=j(x,t),\;\forall(x,t)\in \p\Om\times[0,T),
\eea
where $\G>0$ is a constant, $\chi(t),\;i(x),\;j(x,t)$ and $f$ are continuous and $f>0$. Our work also includes the case $f\equiv 1$. The conditions on $H$ {  and $f$ are described later in this section.} In \cite{BL2}, $H$ is the infinity-Laplacian and $f(u)=3u^2$, and in {  \cite{BL3, TR},} $H$ is the $p$-Laplacian and $f(u)=(p-1)u^{p-2}$. 
{  These are contained in this work and, in addition, are included some fully nonlinear operators such 
as the Pucci operators.}
Equations such as (\ref{sec2.001}) {  are of great interest and} have been studied in great detail in the weak solution setting, see the discussions in the works cited in \cite{BL2} and \cite{DE}. {  In this context, a study of large time asymptotic behaviour of viscosity solutions to the equations in \cite{BL2, BL3} appears in \cite{BL4}.}

We now state precisely the conditions placed on $H$ and also state the main results of this work. Let $o$ denote the origin in $\IR^n$. {  On occasions,} we write a point $x\in \IR^n$ 
as $(x_1,x_2,\cdots, x_n)$. Call $S^n$ the set of all real $n\times n$ symmetric matrices.
Let $I$ be the $n\times n$ identity matrix and $O$ the $n\times n$ matrix with all entries being zero. We reserve $e$ to stand for a unit vector in $\IR^n$.

Through out the work we require that $H\in C(\IR^n\times S^n, \IR)$ and $H(p, O)=0,\;\forall p\in \IR^n$. {  We require that $H$ satisfy the following conditions.}

{\bf Condition A (Monotonicity):} The operator $H(p,X)$ is continuous at $p=0$ for any $X\in S^n$ and $H(p, O)=0$, for any $p\in \IR^n$. In addition, for any $X,\;Y\in S^n$ with $X\le Y$,
\eqRef{sec2.1}
H(p,X)\le H(p,Y),\;\;\forall p\in \IR^n.
\ee
Since $H(p, O)=0$, $H(p, X)\ge 0$, for any $p$ and any $X\ge 0$. \quad $\Box$

{\bf Condition B (Homogeneity):} We assume that there are constants $k_1$, a positive real number, and $k_2$, a positive odd integer, such that for any $(p,X)\in \IR^n\times S^n$,
\bea\label{sec2.2}
H(\tht p, X)=|\tht|^{k_1}H(p, X),\;\;\forall \tht\in \IR, \;\;\;\mbox{and}\;\;\;H(p, \tht X)=\tht^{k_2} H(p, X),\;\;\forall \tht>0.
\eea
{  Define
\eqRef{sec2.3}
k=k_1+k_2\;\;\;\mbox{and}\;\;\;\g=k_1+2k_2.
\ee
While our work allows $k_2\ge 1$ (consistent with Condition A), we consider, mainly, the case $k_2=1$ implying $k=k_1+1$ and $\g=k_1+2$.} $\Box$
\vsp
{  Before stating the third condition, we introduce the following quantities.} Observe that $(e\otimes e)_{ij}=e_ie_j$ and $e\otimes e$ is a non-negative definite matrix.
 For every $-\infty<\lam<\infty$, 
we set
\bea\label{sec2.4}
&&m_{min}(\lam)=\min_{|e|=1}H\left(e,I-\lam e\otimes e\right), \;\;\;m_{max}(\lam)=\max_{|e|=1}H\left(e,I-\lam e\otimes e\right),
\\
&&\mu_{min}(\lam)=\min_{|e|=1}H(e, \lam e\otimes e-I)\;\;\;\mbox{and}\;\;\;\mu_{max}(\lam)=\max_{|e|=1}H(e, \lam e\otimes e- I). \nonumber
\eea
By (\ref{sec2.1}), the functions $m_{min}(\lam)$ and $m_{max}(\lam)$ are {  non-increasing} in $\lam$ while 
$\mu_{min}(\lam)$ and $\mu_{max}(\lam)$ are {  non-decreasing} in $\lam$. 

If $\lam\le1$ then $I-\lam{   e} \times e$ is a {  non-negative} definite matrix and, {  by Condition A},
$m_{max}(\lam)\ge m_{min}(\lam)\ge 0$. Also, 
if $H$ is odd in $X$ then $m_{max}(\lam)=-\mu_{min}(\lam)$ and $m_{min}(\lam)=-\mu_{max}(\lam)$. However, in this work we do not require that $H$ be odd in $X$.

We set 
\eqRef{sec2.6}
m(\lam)=\min\left\{ m_{min}(\lam),\;-\mu_{max}(\lam)\right\}\;\;\mbox{and}\;\;\mu(\lam)=
\max\left\{m_{max}(\lam),\;-\mu_{min}(\lam)\right\}.
\ee
{  Both $\mu(\lam)$ and $m(\lam)$ are non-increasing and $\mu(\lam)\ge m(\lam)\ge 0$, if $\lam\le1$. However, if $\lam> 1$ then $I-\lam e\times e$ is neither non-negative definite nor non-positive definite and it is not clear what signs do $m(\lam)$ and $\mu(\lam)$ have. To address this, we impose a coercivity condition. In Section 3 we have listed several equations that satisfy the condition including Trudinger's equation and equations involving the Pucci operators and the infinity-Laplacian.}

{\bf Condition C (Coercivity):} We take $H$ to be coercive in the following sense. 
We impose that there are $\lam_0$ and $\lam_1$ such that $-\infty< 0<\lam_1\le 1\le \lam_0<\infty$ and
\eqRef{sec2.5}
\mbox{(i)}\;\;m(\lam)>0,\;\;\forall \lam\le \lam_1,\;\;\mbox{and}\;\;\mbox{(ii)}\;\;\mu(\lam)<0,\;\;\forall \lam \ge \lam_0. \quad \Box
\ee
{  Note that this requires $H(e, I-\lam e\otimes e)$, as a function of $\lam$, to change sign in $(-\infty, \infty)$. As noted above, the value $\lam=1$ arises from the observation that
$I-\lam e\otimes e$ changes behaviour at $\lam =1$. As it is seen later the quantities $m(\lam)$ and $\mu(\lam)$ play a significant role in this work in obtaining bounds and estimates for the auxiliary functions that are used in the construction of sub-solutions and super-solutions, see Remark \ref{sec3.120}. Also, see below.}
\vsp
{  In the rest of the work, we distinguish between the following two cases that arise in (\ref{sec2.5})(ii).}
\bea\label{sec2.8}
&&\mbox{Case (i): there is a $\bar{\lam}$ such that $1<\bar{\lam}<2$ such that $\mu(\bar{\lam})<0$,}\nonumber \\
&&\mbox{Case (ii): there is a $\bar{\lam}\ge 2$ such that $\mu(\lam)<0,\;\forall \lam>\bar{\lam}.$}
\eea
The quantity $\bar{\lam}$ in Case (ii) is {   assumed to be minimal in the sense that $\mu(\lam)\ge 0,$ if $\lam<\bar{\lam}.$ The value of $\bar{\lam}$ influences greatly the construction of the sub-solutions and the super-solutions in Sections 5, 6 and 7. In particular, see (\ref{sec6.7}), (\ref{sec7.4}) and (\ref{sec8.4}). Also see (\ref{sec9.1}) in the Appendix, where a version of the weak maximum principle is derived for the class of equations under consideration.}
\vsp
{  Next, we make an observation regarding an operator $\hat{H}$ closely related to $H$.} Define 
$\hat{H}(p,X)=-H(p, -X),\;\forall(p,X)\in \IR^n\times S^n$. 
\begin{rem}\label{sec2.9} It is clear that $\hat{H}$ satisfies Conditions A and B, see (\ref{sec2.1}) and (\ref{sec2.2}). Next, using definitions analogous to (\ref{sec2.4}) and calling $\hat{m}_{min},\;\hat{m}_{max},\;\hat{\mu}_{min}$ and $\hat{\mu}_{max}$ the corresponding quantities for $\hat{H}$, we find that
\ben
&&\hat{m}_{min}(\lam)=-\mu_{max}(\lam),\;\;\;\;\;\hat{m}_{max}(\lam)=-\mu_{min}(\lam),\;\;\;\;\;
\hat{\mu}_{min}(\lam)=-m_{max}(\lam)\\
&&\mbox{and}\;\;\hat{\mu}_{max}(\lam)=-m_{min}(\lam),\;\;\;\forall \lam\in \IR.
\een
It is clear that $\hat{m}(\lam)=m(\lam)$ and $\hat{\mu}(\lam)=\mu(\lam)$. Thus, $\hat{H}$ satisfies Condition C or (\ref{sec2.5}). $\Box$
\end{rem}

From hereon, we define
$$h(x,t)=\left\{\begin{array}{lcr} i(x),  \qquad\;\;\;\forall x\in \Om,\;\mbox{at $t=0$,}\\ j(x,t),\qquad  \forall(x,t)\in \p\Om\times[0,T). \end{array}\right.
$$  
We assume that $i(x)$ and $j(x,t)$ are continuous and $h\in C(P_T)$, i.e, $\lim_{x\rightarrow y}i(x)=j(y,0)=
\lim_{(z,t)\rightarrow (y,0^+)}j(z,t)$, {  for any $y\in \p\Om$ and where} $(z,t)\in \p\Om\times(0,T)$.

We now state the main results of the work. {  Recall (\ref{sec2.2}), (\ref{sec2.3}), $k=k_1+k_2$ and $\g=k_1+2k_2$.}
 
\begin{thm}\label{I} Let $H$ satisfy Conditions A, B and C and $0<T<\infty$. {  Suppose} that $\chi:[0,T]\rightarrow \IR$ and $f:[0,\infty)\rightarrow \IR,\;f>0$ are continuous.
Assume further that Case(i) of (\ref{sec2.8}) holds and $\Om\subset \IR^n,\;n\ge 2,$ is any bounded domain. 

I. Let $k>1$ and $h>0$. Suppose that {  $f$ is an increasing $C^1$ function} and $f^{1/(k-1)}$ is concave. {  Then the problem 
$$H(Du, D^2u)+\chi(t)|Du|^k-f(u)u_t=0,\;\mbox{in $\Om_T$ and $u=h$ in $P_T$},$$
admits a unique positive solution $u\in C(\Om_T\cup P_T)$.}

II. Let $k\ge 1$. If $0<\G<\g,$ then, {  for any continuous function $h$, the following equation
$$H(Du, D^2u)+\chi(t)|Du|^\G-u_t=0,\;\mbox{in $\Om_T$ and $u=h$ in $P_T$},$$
 admits a unique solution $u\in C(\Om_T\cup P_T)$. }
\end{thm}

\begin{thm}\label{II} Let $H$ satisfy Conditions A, B and C and $0<T<\infty$. Assume that $\chi:[0,T]\rightarrow \IR$ and $f:[0,\infty)\rightarrow \IR,\;f> 0$ are continuous. 
Assume further that Case(ii) of (\ref{sec2.8}) holds and $\Om\subset \IR^n,\;n\ge 2,$ is a bounded domain that satisfies a uniform exterior ball condition. 

I. Let $k>1$ and $h>0$. {  Suppose that $f$ is an increasing $C^1$ function} and $f^{1/(k-1)}$ is concave. {  Then the problem 
$$H(Du, D^2u)+\chi(t)|Du|^k-f(u)u_t=0,\;\mbox{in $\Om_T$ and $u=h$ in $P_T$},$$}
admits a unique positive solution $u\in C(\Om_T\cup P_T)$.

II. Let $k\ge 1$. If $0<\G<\g$ then, {  for any continuous $h$, the following problem
$$H(Du, D^2u)+\chi(t)|Du|^\G-u_t=0,\;\mbox{in $\Om_T$ and $u=h$ in $P_T$},$$}
admits a unique solution $u\in C(\Om_T\cup P_T)$. 
\end{thm}
{  In Theorems \ref{I} and \ref{II}, the part I's address the doubly nonlinear case. The part II's require that $\G<\g$ in the case $f\equiv 1$. This restriction can be relaxed to include $\G=\g$ for some equations that can be 
converted by a transformation to a doubly nonlinear case to which Part I applies.} 

{  To illustrate the point, we take an example like Trudinger's equation, i.e, take 
in Theorems \ref{I} and \ref{II}, $k_1=p-2,\;p\ge 2$, $k_2=1$, $k=p-1$, $\g=p$ and $f(u)=u^{p-2}$,
\ben
\mbox{div}\left(|Du|^{p-2} Du\right)+\chi(t)|Du|^{p-1}-(p-1)u^{p-2}u_t=0,\;u>0.
\een
The Part I's of the theorems imply existence. If we make a change of variables $v=\log u$ (see Lemma \ref{sec3.15}) we get
\ben
\mbox{div}\left(|Dv|^{p-2} Dv\right)+(p-1)|Dv|^p+\chi(t)|Dv|^{p-1}-(p-1)v_t=0,
\een
where $v$ can have any sign. Although the Part II's do not apply here we do get existence and uniqueness.}

We prove both parts I and II by taking $h>0$. In part II, since adding constants to a solution yields a solution
we get the claim for any $h$. The concavity of $f^{1/(k-1)}$ is required for a comparison principle to hold, see Section 4, and it is not clear to us if a version of the comparison principle holds if the condition fails to hold. The proof of existence employs the Perron method and a substantial part of the work
is devoted to the construction of appropriate sub-solutions and super-solutions. These are so done that
they are close to the boundary data $h$ in $P_T$ in a local sense. Section 5 contains the details for the initial data while Sections 6 and 7 have details for the side condition. We also remark that some of our results hold for more general operators $H$. However, to keep our presentation clear, we have taken $H$ to be as described above and made remarks and comments along the way where needed. 

We point out that the work in \cite{D} also addresses issues that overlap with our work. 
In \cite{D}, besides homogeneity, $H$ satisfies $\forall (x, p, Y)\in \Om\times\IR^n\times S^n,$ 
\eqRef{sec2.9}
a |p|^{k_1}Trace(X)\le H(x,p, Y+X)-H(x,p, Y)\le b |p|^{k_1} Trace(X),\;\;\mbox{$\forall X\in S^n,\;X\ge 0$,} 
\ee
where $0<a\le b<\infty$ and $k_1>-1$. Thus, $k=k_1+1$ and $\g=k_1+2$. The author considers equations of the type
\eqRef{sec2.90}
H(x, Du, D^2u)+\langle\chi(t), Du\rangle|Du|^{k_1}-u_t=g(x,t),\;\mbox{in $\Om_T$ and $u=h$ in $P_T$.}
\ee
where $H$ and $\chi$ satisfy additional conditions in $x$ and in $t$. The work contains a comparison principle and regularity results under further conditions on $g$  and $h$. The author also shows existence of solutions of the above in domains with exterior cone condition. Clearly, singular cases are also included. We direct the reader to the work for a more detailed discussion. 

We now compare and contrast \cite{D} with the current work. The condition in (\ref{sec2.9}) implies that
\bea\label{sec2.60}
&&(i)\;\;a(t-s) \le H(x,e, I-s e\otimes e)-H(x, e, I-t e\otimes e)\le b (t-s),\;t\ge s,\;\;\mbox{and}  \nonumber\\
&&(ii)\;a\le \frac{H(x,e,I-e\otimes e)}{n-1}\le b.
\eea
Our conditions require that $H(p, X+Y)\ge H(p,X)$, for $Y\ge 0$, and coercivity as stated in condition C. Thus, $H(e, I-se\otimes e)$ is continuous and non-increasing in $s$ (see condition A) and (\ref{sec2.2}) and (\ref{sec2.5}) hold. {  The conditions in (\ref{sec2.2}) and (\ref{sec2.5})} are also satisfied by the operators in \cite{D}. However, we do not require that $H$ be Lipschitz continuous, see (\ref{sec2.60})(i). Also, unlike (\ref{sec2.60})(ii), we allow
the possibility that $H(e, I-e\otimes e)=0$, {  as in the case of the infinity-Laplacian which is a very degenerate operator. In addition, the class of operators $H$ includes
some fully nonlinear operators such as the Pucci operators (as does \cite{D}).} See Section 3 for examples.
{  Equally importantly,} our work addresses the doubly nonlinear case where $f(u)\not\equiv 1$.  The second term involving the gradient, in the doubly nonlinear case, has the same power as in (\ref{sec2.90}). However, we allow a greater range of powers if $f\equiv 1$, see Theorems \ref{I} and \ref{II}. Equations of the kind discussed following the statements of Theorems \ref{I} and \ref{II}, involving two terms in $|Du|$ with differing powers are also included here.  
 
{  On the other hand, our work takes $g=0$ (see (\ref{sec2.90})) and while Theorem \ref{I} applies to any general domain, Theorem \ref{II} is proven for domains with exterior ball condition. We do not address any regularity results and 
the operator $H$ does not depend on $x$ although the results here would hold (modifying the definitions appropriately) if it depended on $t$. }

We describe the layout of the paper. Section 2 contains additional notations, definitions and some auxiliary results. Sections 3 lists examples of $H$ covered by the work. We prove various versions of the comparison principle in Section 4. Sections 5, 6 and 7 provide details of the constructions of the sub-solutions 
and super-solutions and lead to the proofs of Theorems \ref{I} and \ref{II}. These lead to the existence of a unique solution by using Perron's method. In the Appendix, we have included a version of the weak maximum principle for {  (\ref{sec2.001}).}
\vsp
{  We thank the referees for reading the work and for their many suggestions that have helped improve the work.}
\vsp
\section{\bf Notations, definitions and preliminary results}     

Through out this work, $\Om\subset \IR^n,\;n\ge 2,$ is a bounded domain and $\p\Om$ its boundary. For $0<T<\infty$, we define the cylinder
\eqRef{sec3.1}
\Om_T=\Om\times (0,T)=\{(x,t)\in \IR^n\times \IR:\;x\in \Om,\;0<t<T\}.
\ee
The parabolic boundary of $\Om_T$, denoted by $P_T$, is the set 
\eqRef{sec3.2}
P_T=(\Om\times \{0\}) \cup (\p \Om\times [0,T) ).
\ee

Let $B_r(x)\subset \IR^n$ be the ball of radius $r$, centered at $x$. For $r>0$ and $\tau>0$, we define the following open cylinder
\eqRef{sec3.3}
D_{r,\tau}(x,t)=B_r(x)\times \left(t-\tau, t+\tau \right).
\ee  

Our goal in this work is to show existence of positive solutions of {  (\ref{sec2.001}), that is, }
\eqRef{sec3.30}
H(Du, D^2u)+\chi(t)|Du|^{\G}  -f(u)u_t=0,\;\;\mbox{in $\Om_T$, and $u=h$, in $P_T$,}
\ee
where $\chi:[0,T]\rightarrow \IR$ is continuous, {  $f$ is $C^1$ and $f>0,$ and $\G\ge 0$. Also,}
$$h(x,t)=\left\{\begin{array}{lcr} i(x),  \qquad\;\;\;\forall x\in \Om,\;\mbox{at $t=0$,}\\ j(x,t),\qquad  \forall(x,t)\in \p\Om\times[0,T). \end{array}\right.
$$  
We assume that $i(x)$ and $j(x,t)$ are continuous and $h\in C(P_T)$, i.e,  $\lim_{x\rightarrow y}i(x)=j(y,0)=
\lim_{(z,t)\rightarrow (y,0^+)}j(z,t)$, where $y\in \p\Om$ and $(z,t)\in \p\Om\times(0,T)$.

For a set $A\subset \IR^{n+1}$, the function class $usc(A)$ is the set of all functions that are upper semi-continuous on $A$. Similarly, $lsc(A)$ is the set of all functions that are lower semi-continuous on $A$.
\vsp
We discuss the notion of a viscosity sub-solution and a super-solution of the parabolic equation
\eqRef{sec3.4}
H(Dw, D^2w)+\chi(t)|Dw|^\G-f(w)w_t=0,\;\;\mbox{in $\Om_T$.}
\ee
For these definitions, we assume that $H$ satisfies Condition A, see (\ref{sec2.1}), and $f$ is 
a continuous function of one variable and $f>0$. 

Through out this work, by a test function $\psi$ we mean a function that is 
$C^2$ in $x$ and $C^1$ in $t$. 

We say that $u\in usc(\Om_T)$ is a sub-solution of (\ref{sec3.4}) in $\Om_T$
if, for any test function $\psi$, $u-\psi$ has a maximum at a point $(y,s)\in \Om_T$, we have
\eqRef{sec3.5}
H(D\psi(y,s), D^2\psi(y,s))+\chi(s)|D\psi(y,s)|^\G-f(u(y,s))(\psi_t)(y,s)\ge 0.
\ee
In this case, we write $H(Du, D^2u)+\chi(t)|Du|^\G-f(u)u_t\ge 0$. A function $v\in lsc(\Om_T)$ is a super-solution of 
(\ref{sec3.4}) in $\Om_T$ if, for any test function $\psi$, $v-\psi$ has a minimum at a point $(y,s)\in \Om_T$, we have
\eqRef{sec3.5}
H(D\psi(y,s), D^2\psi(y,s))+\chi(s)|D\psi(y,s)|^\G-f(v(y,s))(\psi_t)(y,s)\le 0.
\ee
In this case, we write $H(Dv, D^2v)+\chi(t)|Dv|^\G-f(v)v_t\le 0.$ If $u$ is a sub-solution and a super-solution of (\ref{sec3.4}) then $u\in C(\Om_T)$ and is a solution of (\ref{sec3.4}) in $\Om_T$.

Next, $u$ is a sub-solution of (\ref{sec3.30}) if $u\in usc(\Om_T\cup P_T)$, $u$ is a sub-solution of (\ref{sec3.4}) 
and $u\le h$ in $P_T$. Similarly, $u$ is a super-solution of (\ref{sec3.30}) if $u\in lsc(\Om_T\cup P_T)$, 
$u$ is a super-solution of
(\ref{sec3.4}) and $u\ge h$ in $P_T$. We say $u$ is a solution of (\ref{sec3.30}) if $u\in C(\Om_T\cup P_T)$, $u$ is a solution of 
(\ref{sec3.4}) and $u=h$.
\vsp
In this work, we construct sub-solutions and super-solutions {  that are $C^2$} functions of $x$ and $t$. With (\ref{sec3.30}) in mind, we state
an expression for the operator $H$ and this will be applied quite frequently in this work. Let $\s(t)>0$ and $v(x)$ be a $C^2$ function. Using (\ref{sec2.2}) and (\ref{sec2.3}),
\eqRef{sec3.51}
H(D \s v, D^2 \s v)=\s^k H(Dv, D^2v).
\ee

Let $v(x)=v(r)$ where $r=|x-z|$, for some $z\in \IR^n$. Set $e=(e_1,e_2,\cdots,e_n)$ where $e_i=(x-z)_i/r,\;\forall i=1,2,\cdots, n$. Then for $x\ne z$, 
\eqRef{sec3.6}
 H(Dv, D^2v+dDv\otimes Dv)=H\left(v^{\prime}(r)e, \left(\frac{v^{\prime}}{r}\right)I+\left(v^{\prime\prime}+d\left( v^{\prime}\right)^2-\frac{v^{\prime}}{r}\right) e\otimes e \right),
\ee
where $I$ is the $n\times n$ identity matrix and $d=0$ or $1$. We now take $d=0$ and use Condition B. If $v^{\prime}\ge 0$ then (\ref{sec3.6}) shows that
\eqRef{sec3.7}
H(Dv, D^2v)=\frac{(v^{\prime})^k}{r^{k_2}}H\left(e,\; I+\left(\frac{ rv^{\prime\prime}}{v^{\prime}}-1\right) e\otimes e \right).
\ee
If $v^{\prime}\le 0$ then (\ref{sec3.6}) leads to
\eqRef{sec3.8}
H(Dv, D^2v)=\frac{|v^{\prime}|^k}{r^{k_2}}H\left(e,\; -\left( I+\left(\frac{ rv^{\prime\prime}}{v^{\prime}}-1\right) e\otimes e\right) \right),
\ee

We apply (\ref{sec3.7}) and (\ref{sec3.8}) to the function $v(r)=a+b r^{\beta}$ where $a+br^{\beta}>0$. We note
\bea\label{sec3.9}
\frac{ rv^{\prime\prime}}{v^{\prime}}-1=\beta-2.
\eea
Using (\ref{sec3.7}), (\ref{sec3.9}) and recalling that $k=k_1+k_2$ and $\g=k_1+2k_2$ (see (\ref{sec2.3})), we get
\bea\label{sec3.10}
H(Dv, D^2v)&=&r^{-k_2}\left( b\beta r^{\beta-1} \right)^k H(e, \; I+(\beta-2)e\otimes e)\nonumber\\
&=&\left( b\beta\right)^k r^{\beta k-\g} H(e,\; I-(2-\beta)e\otimes e),\;\;\;\;\;\mbox{if $b\beta>0$.}
\eea
Similarly, using (\ref{sec3.8}) and (\ref{sec3.9}), we get
\bea\label{sec3.11}
H(Dv, D^2v+Dv\otimes Dv)&=&r^{-k_2}\left( |b\beta| r^{\beta-1} \right)^k H(e, \; -I-(\beta-2)e\otimes e)\nonumber\\
&=&\left( |b\beta|\right)^k r^{\beta k-\g} H(e,\; (2-\beta)e\otimes e-I),\;\;\;\;\;\mbox{if $b\beta<0$.}
\eea
\vsp
\begin{rem}\label{sec3.12} In this work, we take $d=0$ and we make use of (\ref{sec3.10}) and (\ref{sec3.11}) in Sections 5, 6 and 7. 

The expressions in (\ref{sec3.6})-(\ref{sec3.11}) hold if $H$ depends on $t,\;u, \;Du$ and $D^2u$. However, to keep our exposition clearer, we will take $H$ to depend on $Du$ and $D^2u$ and make comments about more general situations as and when the need arises.  $\Box$
\end{rem}

\begin{rem}\label{sec3.120}
Recall (\ref{sec2.6}), (\ref{sec3.10}) and (\ref{sec3.11}). Let $v=a+br^\beta$ then the following hold.
$$\mbox{(i)}\qquad\frac{ (b \beta )^k m(2-\beta)}{r^{\g-\beta k}} \le H(Dv, D^2v)\le \frac{\left( b\beta\right)^k \mu(2-\beta)}{r^{\g-\beta k}},\;\;\;\mbox{if $b\beta>0$.}$$

$$\mbox{(ii)}\quad -\frac{\left( |b\beta|\right)^k \mu(2-\beta)}{r^{\g-\beta k}}\le H(Dv, D^2v)\le 
-\frac{\left( |b\beta|\right)^k m(2-\beta)}{r^{\g-\beta k}},\;\;\;\mbox{if $b\beta<0$.}$$

We make use of the above {  estimates} in Sections 5, 6 and 7.    \quad $\Box$
\end{rem}

We now discuss a change of variables formula needed for a version of the comparison principle for equations of the kind
$$H(Du, D^2u)+\chi(t)|Du|^{\G}-f(u)u_t=0,$$
where $\chi:[0,T]\rightarrow \IR$ is continuous. {  Recall from (\ref{sec2.2}) and (\ref{sec2.3}) that $k=k_1+k_2$ and $\g=k_1+2k_2$.}
In this work, we take (a) $\G=k$ for a non-constant $f$ and $k>1$, and (b) any $0<\G<\g$ for $f\equiv 1$ and $k\ge 1.$ 

Let $f:\IR \rightarrow \IR$ {  be a $C^1$ function and $f>0$.}
For $k>1$,
define $\phi:\IR\rightarrow \IR$ to be a $C^2$ solution of
\bea\label{sec3.13}
\frac{d\phi}{d\tau}=\left\{(f\circ \phi)(\tau)\right\}^{1/(k-1)}
\eea
Thus, $\phi$ is increasing. For proving the comparison principle in Section 4, we will assume further that
\eqRef{sec3.14}
\mbox{$f^{1/(k-1)}$ is concave, i.e,}\;\left\{f^{1/(k-1)}\right\}^{\prime}{ (\tau)}\;\;\mbox{is non-increasing in { $\tau$}. }
\ee
Combining (\ref{sec3.13}) and (\ref{sec3.14}) the above reads
$$\frac{d\log \phi^{\prime}(\tau)}{d\tau}=
\frac{\phi^{\prime\prime}(\tau)}{\phi^{\prime}(\tau)}=\left[ \frac{f^{\prime}(\phi(\tau))}{(k-1)}\right]
\left\{(f\circ \phi)(\tau)\right\}^{(2-k)/(k-1)}\;\;\mbox{is non-increasing in $\tau$.}$$
The facts that $f$ is positive and $f^{1/(k-1)}$ is concave impose restrictions on the domain of $f$. From hereon, for all the main results we take 
$$\mbox{$f$ is defined on $[c,\infty)$, $c\ge 0$, $f>0$ and $f$ is increasing.}$$

We now prove the following change of variables lemma. We do this for a somewhat more general case and do not require that (\ref{sec3.14}) hold. 

\begin{lem}\label{sec3.15} Let $H$ satisfy Conditions A and B, see (\ref{sec2.1}) and (\ref{sec2.2}), {  $f:[0,\infty)\rightarrow \IR^+$ be a $C^1$ function and
$g:\Om\times\IR\times \IR\rightarrow \IR$} and $\chi:[0,T]\rightarrow \IR$ be continuous.

{  Let $k=k_1+k_2$ (see (\ref{sec2.3})) and $\phi:\IR\rightarrow \IR$ is a positive $C^2$ increasing function. Set $\tilde{f}(v)=\left\{(f\circ \phi)(v)\right\}^{k/(k-1)}$ and $\tilde{g}(x,t,v)=g(x,t, \phi(v))$.}

Case (i):  {  Suppose that $k>1$ and $\phi$ is as in (\ref{sec3.13}).} We assume that $f$ is non-constant.

(a) If $u\in usc(\Om_T),\;u>0,$ solves $H(Du, D^2u)+\chi(t)|Du|^{k}+g(x,t,u)\ge f(u) u_t$ in $\Om_T$ and $v=\phi^{-1}(u)$ then $v \in usc(\Om_T)$ and
$$H\left(Dv, D^2v+\frac{\phi^{\prime\prime}(v)}{\phi^{\prime}(v)}Dv\otimes Dv \right)+\chi(t)|Dv|^{k}+\frac{\tilde{g}(x,t, v)}{\tilde{f}(v)}\ge v_t,\;\;\mbox{in $\Om_T$}.$$
The converse also holds.

(b) If $u\in lsc(\Om_T),\;u>0,$ solves $H(Du, D^2u)+\chi(t)|Du|^{k}+g(x,t,u)\le f(u) u_t$ in $\Om_T$ and $v=\phi^{-1}(u)$ then $v \in lsc(\Om_T)$ and
$$H\left(Dv, D^2v+\frac{\phi^{\prime\prime}(v)}{\phi^{\prime}(v)}Dv\otimes Dv \right)+\chi(t)|Dv|^{k}+\frac{\tilde{g}(x,t,v)}{\tilde{f}(v)}\le v_t,\;\;\mbox{in $\Om_T$},$$
and conversely.

Case (ii): Let $k=1$. If $f\equiv 1$ then the claims in (a) and (b) hold if $\phi(\tau)$ is any increasing positive $C^2$ function {  (define $\tilde{f}\equiv 1$)}. In particular, if $\phi(\tau)=e^\tau$ and {  $u\in usc(\Om_T)$} then
$H(Du, D^2u)+\chi(t)|Du|+g(x,t,u)-u_t\ge (\le) 0$ if and only if 
$$H(Dv, D^2v+Dv\otimes Dv)+\chi(t)|Dv|+\frac{\tilde{g}(x,t, v)}{\phi^{\prime}(v)}-v_t\ge (\le) 0.$$
\end{lem}
\begin{proof} We prove Case (i) part (a) and start with the converse. Let $\phi$ be as in (\ref{sec3.13})
and $v\in usc(\Om_T)$ solve
$$H\left(Dv, D^2v+\frac{\phi^{\prime\prime}(v)}{\phi^{\prime}(v)}Dv\otimes Dv\right)+\chi(t)|Dv|^k+\frac{\tilde{g}(x,t,v)}{\tilde{f}(v)}-v_t\ge 0.$$

Take $u=\phi(v)$ and let $\psi$ to be a test function such that $u-\psi$ has a maximum at a point $(y,s)\in \Om_T$. Clearly,
\ben
v(x,t)\le \phi^{-1}\left(\;u(y,s)+\psi(x,t)-\psi(y,s)\;\right) ,\;\;\forall(x,t)\in \Om_T.
\een
\vsp
Calling $\zeta(x,t)=\phi^{-1}\left(\;u(y,s)+\psi(x,t)-\psi(y,s)\;\right)$, we get $(v-\zeta)(x,t)\le (v-\zeta)(y,s)=0.$
Thus, $v-\zeta$ has a maximum at $(y,s)$ and hence,
\eqRef{sec3.16}
H\left(D\zeta(y,s),  \left(D^2\zeta+\frac{\phi^{\prime\prime}(v)}{\phi^{\prime}(v)}D\zeta\otimes D\zeta\right)(y,s) \right)+{ \chi(s)}|D\zeta(y,s)|^k+   \frac{\tilde{g}(y,s,v(y,s))}{\tilde{f}(v(y,s))}-\zeta_t(y,s)\ge 0. 
\ee
We note
\ben
&&D\zeta(y,s)=\frac{D\psi(y,s)}{\phi^{\prime}(\zeta(y,s))},\;\;\;\;\;\;\zeta_t(y,s)=\frac{\psi_t(y,s)}{\phi^{\prime}(\zeta(y,s))}\;\;\;\mbox{and}\\
&&D^2\zeta(y,s)=\frac{D^2\psi(y,s)}{\phi^{\prime}(\zeta(y,s))}-\left[ \frac{\phi^{\prime\prime}(\zeta(y,s))}
{(\phi^{\prime}(\zeta(y,s))}\right]D\zeta(y,s)\otimes D\zeta(y,s).
\een
Recalling that $\zeta(y,s)=v(y,s)$ and using the above, we get
\eqRef{sec3.17}
\frac{D^2\psi(y,s)}{\phi^{\prime}(v(y,s))}=D^2\zeta(y,s)+\frac{\phi^{\prime\prime}(v(y,s))}{\phi^{\prime}(v(y,s))}D\zeta(y,s)\otimes D\zeta(y,s).
\ee

Using (\ref{sec2.2}), (\ref{sec3.17}) and the definitions of $\tilde{f}$ and $\tilde{g}$, we get from (\ref{sec3.16})
\ben
&&0\le H\left(\frac{D\psi(y,s)}{\phi^{\prime}(v(y,s))},\; \frac{D^2\psi(y,s)}{\phi^{\prime}(v(y,s))} \right)+\chi(s)\left( \frac{|D\psi(y,s)|}{\phi^{\prime}(v(y,s))}\right)^k+\frac{g(y,s,u(y,s))}{  \left\{ f(u(y,s))\right\}^{k/(k-1)}}-\frac{\psi_t(y,s)}{\phi^{\prime}(v(y,s))},\\
&&=\frac{H\left(D\psi(y,s),  \;D^2\psi(y,s)\right)}{ \left\{\phi^{\prime}(v(y,s))\right\}^{k}}+\chi(s)\left( \frac{|D\psi(y,s)|}{\phi^{\prime}(v(y,s))}\right)^k+ \frac{g(y,s,u(y,s))}{ \left\{ f(u(y,s)) \right\}^{k/(k-1)}} - \frac{\psi_t(y,s)}{\phi^{\prime}(v(y,s))}.
\een
Using (\ref{sec3.13}), we get $H(Du, D^2u)+\chi(t)|Du|^k+g(x,t,u)-f(u)u_t\ge 0.$

Suppose that $u\in usc(\Om_T)$ solves $H(Du, D^2u)+\chi(t)|Du|^k+g(x,t,u)-f(u)u_t\ge 0$. Define
$v=\phi^{-1}(u)$. 

Let $\psi$ be a test function such that $v-\psi$ has a maximum at $(y,s)$, i.e, 
$v(x,t)\le v(y,s)+\psi(x,t)-\psi(y,s),\;\forall(x,t)\in \Om_T.$ Thus,
$$u(x,t)\le \phi(\; v(y,s)+\psi(x,t)-\psi(y,s)\;),\;\;\forall(x,t)\in \Om_T.$$
Let $\eta(x,t)=\phi( \;v(y,s)+\psi(x,t)-\psi(y,s)\; )$ implying that $\eta(y,s)=u(y,s)$,
$(u-\eta)(x,t)\le (u-\eta)(y,s)=0$ and 
$$H(D\eta(u,s), D^2\eta(y,s))+\chi(s)|D\eta(y,s)|^k+g(y,s,u(y,s))-f(u(y,s))\eta_t(y,s)\ge 0$$
Calculating,
\ben
&&0\le H(D\eta(y,s), D^2\eta(y,s))+\chi(s)|D\eta(y,s)|^k+g(y,s,u(y,s))-f(u(y,s))\eta_t(y,s)\\
&&=[\phi^{\prime}(v(y,s))]^kH\left(D\psi(y,s), D^2\psi(y,s)+{  \frac{\phi^{\prime\prime}(v(y,s))}{\phi^{\prime}(v(y,s))} }D\psi(y,s)\otimes D\psi(y,s) \right)\\
&&+\chi(s)[\phi^{\prime}(v(y,s))]^k|D\psi(y,s)|^k+g(y,s, (\phi\circ v)(y,s))-(f\circ\phi\circ v)(y,s) \phi^{\prime}(v(y,s))\psi_t.
\een
Simplifying, we see that the claim holds. The claims in Case (i) (b) and Case (ii) follow analogously. 
\end{proof}
\vsp
\begin{rem}\label{sec3.170} (i) Lemma \ref{sec3.15} does not address the case $f\equiv 1$ and $k>1$ since the comparison principle 
for $H(Du, D^2u)+\chi(t)|Du|^\G-u_t=0$, where $\G\ge 0$, follows from a general result. See Section 4.

{ (ii) We now address the example that was referred to in the discussion following Theorem \ref{II}, see Section 1. Let
$Tr(X)$ be the trace of a matrix $X$. Set 
$$H(p,X)=|p|^{q-2}Tr(X)+(q-2)|p|^{q-4} p_ip_jX_{ij},\;\;q\ge 2.$$
Clearly, $H(Du, D^2u)=$div$(|Du|^{q-2}Du)$. If $X=Y+p\otimes p$ then
$$H(p, Y+p\otimes p)=|p|^{q-2}Tr(Y)+(q-2)|p|^{q-4} p_ip_jY_{ij}+(q-1)|p|^q=H(p,Y)+(q-1)|p|^q.$$
Suppose that $u>0$ solves div$(|Du|^{q-2}Du)+\chi(t)|Du|^{q-1}-(q-1)u^{q-2}u_t=0.$ It follows from (\ref{sec3.13}),
$\phi(s)=e^s$. If $v=\log u$ then Lemma \ref{sec3.15} and the above observations imply that
$$\mbox{div}(|Dv|^{q-2}Dv)+(q-1)|Dv|^q+\chi(t)|Dv|^{q-1}-(q-1)v_t=0.$$
Thus, showing the existence of $u$ is equivalent to showing the existence of $v$. See Section 3. $\Box$ }
\end{rem} 
  
\begin{rem}\label{sec3.18} It is clear from Lemma \ref{sec3.15} that analogous results hold if 
$H$ satisfies Condition $A$ and $B$ and depends on $x,\;t,\;u,\;Du$ and $D^2u$. $\Box$
\end{rem}

Finally, we state a lemma that will be used Sections 5, 6 and 7. Note that the result holds if $H$ depends on $t, \;u, \;Du,\;D^2u$ and $H(t,u, Du, O)=0$.   

\begin{lem}\label{sec3.21} {  Let $O\subset \Om_T$ be a sub-domain. Suppose that $\ell:\IR\rightarrow \IR$, $\chi:[0,T]\rightarrow \IR$ and $f:\IR\rightarrow \IR$ are continuous. Assume that $H$ satisfies Condition A (see (\ref{sec2.1})) and $\G\ge 0$.} Suppose that $u\in usc(lsc)(\Om_T\cup P_T)$ satisfies 
$$H(Du, D^2u+\ell(u)Du\otimes Du)+\chi(t)|Du|^\G-f(u)u_t\ge (\le)0,\;\;\mbox{in $O$}.$$
Assume that for some $c\in \IR$, $u\ge (\le)c$ in $O$, $u=c$ on $\p O\cap \Om_T$, and $u=c$ in $\Om_T\setminus O$. Then $u$ satisfies
$$H(Du, D^2u+\ell(u)Du\otimes Du)+\chi(t)|Du|^\G-f(u)u_t\ge (\le)0,\;\;\mbox{in $\Om_T$}.$$
\end{lem}
\begin{proof} We prove the statement when $u$ is a sub-solution. We check at points on $\p O\cap \Om_T$. 

Let $(y,\tau)\in \p O\cap \Om_T$, with $\tau>0$. Suppose that $\psi$ is a test function such that $u-\psi$ has a maximum at $(y,\tau)$. Since $u\ge c$ and $u(y,\tau)=c$, we have
\ben
0\le u(x,t)-u(y,\tau)&\le& \langle D\psi(y,\tau), x-y\rangle+\psi_t(y,\tau)(t-\tau)+\frac{\langle D^2\psi(y,\tau)(x-y), x-y\rangle}{2}\\
&+&o(|x-y|^2+|t-\tau|),
\een
as $(x,t)\rightarrow (y, \tau).$
Clearly, $D\psi(y,\tau)=0$,
$\psi_t(y,\tau)=0$ and $D^2\psi(y, \tau)\ge 0$. Thus, {  using Condition A,}
\ben
&&H\left(D\psi(y,\tau), D^2\psi(y,\tau)+\ell(u(y,\tau))D\psi(y,\tau)\otimes D\psi(y, \tau)\right)\\
&&\qquad\qquad\qquad+\chi(\tau)|D\psi(y,\tau)|^\G-f(\psi(y,\tau))(\psi_t)(y,\tau)=H(0, D^2\psi(y,\tau) )\ge 0.
\een
The conclusion holds. The proof when $u$ is a super-solution is analogous.
\end{proof}

\section{Examples of $H$}

In this section, we list examples of operators $H$ that satisfy Conditions A, B and C {  and} to which our results apply. 
Let $\lam\in \IR$ and $e\in \IR^n$ be such that $|e|=1$. Set $r=|x|,\;\forall x\in \IR^n$. Recall the definitions of $k_1,\;k_2,\;k,\;m(\lam)$ and $\mu(\lam)$ from
(\ref{sec2.2}), (\ref{sec2.3}), (\ref{sec2.6}) and (\ref{sec2.5}). 
\vsp
{\bf Example 1: The $p$-Laplacian and the pseudo $p$-Laplacian.} 
Recall that the $p$-Laplacian $\D_p$, for $p\ge 2,$ is 
$D_p u=|Du|^{p-2}\D u+(p-2)|Du|^{p-4} \D_\infty u$, where $\D_\infty u=\sum_{i,j=1}^n D_iuD_ju D_{ij}u$ is the infinity-Laplacian. We consider a some what more general version. Define
$$H(Du, D^2u)=|Du|^{q}\D u+a|Du|^{q-2} \D_\infty u,$$ 
where $q\ge 0$ and $a>-1$. {  Then $H(e, I-\lam e\otimes e)=n+a-\lam(1+a).$ Clearly, Conditions A, B and C are met.}

Next we discuss a version of the pseudo $p$-Laplacian, denoted by $\D_{p,q}^s$, where 
$$H(Du, D^2u)=\D_{p,q}^su=|Du|^q\sum_{i=1}^n |D_iu|^p D_{ii}u,\;\;\mbox{where $p,\;q\ge 0$}.$$
Thus, $H(e, I-\lam e\otimes e)=\sum_{i=1}^n |e_i|^p-\lam \sum_{i=1}^n |e_i|^{p+2}$ and $H>0$, if $\lam\le 0.$

Let $\lam>0$. Note that
$H(e, I-\lam e\otimes e)\ge (1-\lam)\sum_{i=1}^n |e_i|^p$, since $|e_i|\le 1$. 
By H$\ddot{\mbox{o}}$lder's inequality, if $r\ge 0$ then
\eqRef{sec4.1}
\min\left(1,\; n^{(2-r)/2}\right)\le \sum_{i=1}^n |e_i|^r\le \left( \sum_{i=1}^n |e_i|^{r+2}\right)^{r/(r+2)} n^{2/(r+2)}.
\ee
Apply (\ref{sec4.1}) with $r=p$ to get a lower bound for $H$, that is,
\ben
H(e, I-\lam e \otimes e)=\sum_{i=1}^n |e_i|^p-\lam \sum_{i=1}^n |e_i|^{p+2}\ge (1-\lam) \sum_{i=1}^n |e_i|^p\ge \left\{\begin{array}{lcr}  (1-\lam)n^{-|2-p|/2},& 0\le \lam\le 1,\\
{  (1-\lam)n},& \lam\ge 1. \end{array}\right.
\een
Set ${  E=E(e)}=(\sum_{i=1}^n|e_i|^{p+2})^{p/(p+2)}$. Use (\ref{sec4.1}) first with $r=p$ and then with $r=p+2$ to get an upper bound for $H$, that is,
\ben
&&H(e, I-\lam e\otimes e)=\sum_{i=1}^n |e_i|^p-\lam \sum_{i=1}^n |e_i|^{p+2}\le \left(\sum_{i=1}^n|e_i|^{p+2}\right)^{p/(p+2)}n^{2/(p+2)}-\lam  \sum_{i=1}^n |e_i|^{p+2}\\
&&\le {  E}\left[ n^{2/(p+2)} -\lam \left(\sum_{i=1}^n |e_i|^{p+2} \right)^{2/(p+2)}\ \right]\le
{  E} \left[ n^{2/(p+2)} -\frac{\lam}{n^{p/(p+2)}} \right]={  E}\left( \frac{n-\lam}{n^{p/(p+2)}} \right)\\
&&=\left(\frac{\sum_{i=1}^n|e_i|^{p+2}}{n}\right)^{p/(p+2)}(n-\lam)\le I(\lam)\left(n-\lam \right),
\een
where $I(\lam)=1$, if $\lam\le n$, and $I(\lam)=n^{-p/2}$, if $\lam\ge n$. Observe that if $e_i=1$, for some $i$, then $H(e, I-\lam e\otimes e)=1-\lam$. Also, if $e_i=n^{-1/2}$, for $i=1,2,3,\cdots, n,$ and then $H(e, I-\lam e\otimes e)=n^{-p/2}(n-\lam).$ Conditions A, B and C hold.
\vsp

{\bf Example 2: The $\infty$-Laplacian and a related operator.} 
Setting $H(Du, D^2u):=\D_\infty u=\sum_{i,j=1}^nD_iu D_juD_{ij}u$, we get
$H(e, I-\lam e\otimes e)=1-\lam. $ 
\vsp
Next, we consider $q\ge 0$ and define
$H(Du, D^2u):=\sum_{i,j=1}^n|D_iu|^q|D_ju|^q D_i u D_j uD_{ij}u.$
Then
\ben
H(e, I-\lam e\otimes e)=\sum_{i=1}^n |e_i|^{2q+2}-\lam \left(\sum_{i=1}^n |e_i|^{q+2} \right)^2.
\een
{  We use (\ref{sec4.1}) for estimating $H(e,I-\lam e\otimes e)$.}
If $\lam \le 0$ then $H> 0$. Taking $\lam\ge 0$ and observing that $(\sum_{i=1}^n |e_i|^{q+2})^2\le \sum_{i=1}^n |e_i|^{2q+2}\le 1$, we get
$$H(e, I-\lam e\otimes e)\ge (1-\lam) \sum_{i=1}^n |e_i|^{2q+2}\ge\left\{\begin{array}{lcr} (1-\lam)n^{-q},&  0\le \lam \le 1,\\ 1-\lam,&  \lam \ge 1.\end{array}\right.$$ 
Noting that $\sum_{i=1}^n |e_i|^{2q+2}\le \sum_{i=1}^n |e_i|^{q+2} $ and using (\ref{sec4.1}), we get
\ben
H(e, I-\lam e\otimes e)\le \sum_{i=1}^n |e_i|^{q+2} \left(1-\lam \sum_{i=1}^n |e_i|^{q+2}\right)\le I(\lam)\left(1-\frac{\lam}{ n^{q/2}} \right),
\een
where $I(\lam)=1$, if $\lam\le n^{q/2}$ and $I(\lam)=n^{-q/2}$, if $\lam\ge n^{q/2}$.
Conditions A, B and C are satisfied. See also \cite{KJ}.
\vsp
{\bf Example 3: Pucci operators.} Let $a_i,\;i=1,2,\cdots, n,$ denote the eigenvalues of the matrix $D^2u$. 

For $0<\theta\le{   \hat \vartheta}$ and $q\ge 0$ define  
\ben
M^{+,q}_{\theta, {   \hat \vartheta}}(u)=|Du|^q\left( {   \hat \vartheta}\sum_{a_i\ge 0} a_i+\theta \sum_{a_i\le 0}a_i\right)\;\;\mbox{and}\;\;M^{-,q}_{\theta, {   \hat \vartheta}}(u)=|Du|^q\left( \theta\sum_{a_i\ge 0} a_i+{   \hat \vartheta} \sum_{a_i\le 0}a_i\right).
\een
For any $e$ with $|e|=1$, the eigenvalues of $I-\lam e\otimes e$ are $1$, with multiplicity $n-1$, and $1-\lam$. Set $H^{\pm}(Du, D^2u)=M^{\pm, q}_{\theta, {   \hat \vartheta}}(u)$ and observe that $H^+(e, \pm(I-\lam e\otimes e))=- H^-(e, \mp(I-\lam e\otimes e)).$ Clearly,
\ben
&&H^+(e, I-\lam e\otimes e)=\left\{ \begin{array}{lcr}\qquad {   \hat \vartheta}(n-\lam), & \lam\le 1,\\ {   \hat \vartheta}(n-1)+\theta(1-\lam),& \lam\ge 1\end{array}\right.
\\
\mbox{and}&& H^-(e, I-\lam e\otimes e)=\left\{ \begin{array}{lcr} \qquad\theta (n-\lam),& \lam\le 1,\\
\theta (n-1)+{   \hat \vartheta}(1-\lam), & \lam\ge 1. \end{array}\right.
\een
Thus, $H^{\pm}$ satisfy Conditions A, B and C. The maximal and minimal Pucci operators are also included here, see \cite{GT}.
\vsp

\section{Comparison principles}

In this section we prove a version of the comparison principle that applies to the class of parabolic equations addressed in the work. {  If $k>1$ and $f$ is an increasing function and 
$f^{1/(k-1)}$ is concave} (the equation is doubly nonlinear) then 
the comparison principle is proven under the condition that sub-solutions and super-solutions are positive. However, if $f\equiv 1$ and $k\ge 1$ then a comparison principle holds without any restrictions on the sign of the sub-solutions and super-solutions. 
 
We {  now} state a comparison principle which is a slight variant of the version in \cite{CIL} and the statement is influenced by the change of variables Lemma \ref{sec3.15}. We consider a more general operator than $H$. Let $F:\IR^+\times \IR\times \IR^n\times S^n\rightarrow\IR$ be continuous and satisfy
\bea\label{sec5.8}
&&\mbox{(i) $F(t,r, p, X)\le F(t,r, p, Y),\;\forall(t,r,p)\in \Om_T\times \IR^n$, and $\forall X,\;Y\in S^n$ with $X\le Y$,}\nonumber\\
&&\mbox{(ii) $\forall (t,p,X)\in \IR^+\times \IR\times S^n$, $F(t, r_1, p, X)\le F(t, r_2, p, Y),$ if $r_1\ge r_2$.}
\eea
In Lemma \ref{sec5.9}, the only condition imposed on $F$ is (\ref{sec5.8}). 
\begin{lem}\label{sec5.9}{(Comparison principle)} Let $F$ be as in (\ref{sec5.8}), $g:\IR\rightarrow \IR$ be a bounded non-increasing continuous function and {  $\kappa:\IR^+\rightarrow \IR^+$} be continuous.
Suppose that $\Om\subset \IR^n$ is a bounded domain and $T>0$. Let $u\in usc(\Om_T\cup P_T)$ and $v\in lsc(\Om_T\cup P_T)$ satisfy in $\Om_T$,
$$F(t,u, Du, D^2u+g(u)Du\otimes Du)- { \kappa(t)} u_t\ge 0\;\;\mbox{and}\;\;F(t,v, Dv, D^2v+g(v)Dv\otimes Dv)-  { \kappa(t)} v_t\le 0.$$
If $\sup_{P_T}v<\infty$ and $u\le v$ on $P_T$ then $u\le v$ in $\Om_T$.
\end{lem}
\begin{proof} We note that $X+g(u)p\otimes p\le Y+g(v)p\otimes p$, for any 
$p\in \IR^n$, $X\le Y$ and $u\ge v$. The claim follows from {   Theorem 33 on page 18 of} \cite{CIL}. 
\end{proof}

\begin{rem}\label{sec5.10} (a) Let $F$, ${  \kappa}$, $u$ and $v$ be as in Lemma \ref{sec5.9}. Let $k=\sup_{P_T}(u-v)^+$ and $v_k=v+k$. Since $v_k\ge v$, by (\ref{sec5.8})(ii),
$$F(t,v_k, Dv_k, D^2v_k+g(v_k)Dv_k\otimes Dv_k)-  { \kappa(t)}(v_k)_t\le 0,\;\mbox{in $\Om_T$ and}\;u\le v_k,\;\mbox{in $P_T$}.$$
By {  Lemma} \ref{sec5.9}, $u-v\le \sup_{P_T}(u-v)^+.$ 

(b) Suppose that $F=F(t, p, X)$ where $p\in \IR^n$ and $X\in S^n$. {  Take $d\ge 0$, a constant.}  Let $u\in usc(\Om_T)$ and $v\in lsc(\Om_T)$ solve
$$F(t, Du, D^2u+dDu\otimes Du )-u_t\ge 0,\;\;\mbox{and}\;\;F(t, Dv, D^2v+dDv\otimes Dv )-v_t{  \le}  0,\;\;\mbox{in $\Om_T$}.$$
Then $u-v\le \sup_{P_T}(u-v).$ To see this, set $k=\sup_{P_T}(u-v)$ and take $v_k=v+k$. Lemma \ref{sec5.9} shows that $u\le v_k$ in $\Om_T$ and the claim holds.
$\Box$ 
\end{rem}
 
As an application of the above result we get a comparison principle for parabolic equations of the type {  (see (\ref{sec2.001}))}
$$H(Du, D^2u)+\chi(t)|Du|^\G-f(u) u_t=0,\;\;\mbox{in $\Om_T$.}$$
Recall (\ref{sec3.13}), (\ref{sec3.14}) and Lemma \ref{sec3.15}.

\begin{thm}\label{sec5.11}{(Comparison principle)} Let $H$ satisfy Conditions A and B, see (\ref{sec2.1}), (\ref{sec2.2}) and (\ref{sec2.3}). Suppose $f:[0,\infty)\rightarrow [0,\infty),$ is a $C^1$ function and $\G\ge 0$.

{  Case (i): $k>1$, $f$ is a non-constant increasing function and $f^{1/(k-1)}(\tht)$ is concave in $\tht$.
Let $u\in usc(\Om_T\cup P_T)$ and $v\in lsc(\Om_T\cup P_T)$ satisfy
$$H(Du, D^2u)+\chi(t)|Du|^k- f(u) u_t\ge 0,\;\;\mbox{and}\;\;H(Dv, D^2v)+\chi(t)|Dv|^k- f(v) v_t\le 0,\;\;\mbox{in $\Om_T$}.$$}
Let $\phi:\IR\rightarrow \IR$ be an increasing $C^2$ function such that $\phi^{\prime}(\tau)=f(\phi(\tau))^{1/(k-1)}$, see (\ref{sec3.13}). 
If $u> 0$, $v> 0$, $\sup_{P_T}v<\infty$ and $u\le v$ on $P_T$ then $\phi^{-1}(u)\le \phi^{-1}(v)$ and $u\le v$ in $\Om_T$. In general,
$$u\le \phi\left( \phi^{-1}(v)+\sup_{P_T}\{\phi^{-1}(u)-\phi^{-1}(v)\}^+\right).$$ 

Case (ii): {  $k\ge 1$ and any $\G\ge 0$. 
Let $u\in usc(\Om_T\cup P_T)$ and $v\in lsc(\Om_T\cup P_T)$ satisfy
$$H(Du, D^2u)+\chi(t)|Du|^\G- u_t\ge 0,\;\;\mbox{and}\;\;H(Dv, D^2v)+\chi(t)|Dv|^\G- v_t\le 0,\;\;\mbox{in $\Om_T$}.$$}
If $u\le v$, in $P_T$ and $\sup_{P_T}v<\infty$ then
$u\le v$ in $\Om_T$. More generally, $u-v\le \sup_{P_T}(u-v).$ The result holds regardless of the signs of $u$ and $v$.
\end{thm}
\begin{proof} The claims follow from the change of variables Lemma \ref{sec3.15}, the comparison principle in Lemma \ref{sec5.9}, Remark \ref{sec5.10} {  and that
$\phi^{\prime\prime}(s)/\phi^{\prime}(s)$ is decreasing in $s$.}
\end{proof}

\begin{rem}\label{sec5.12} (a) Theorem \ref{sec5.11}(i) holds for operators $H$ that depend on $t,\;u, \;Du,\;D^2u$,
$H$ is decreasing in $u$ and satisfy Conditions A and B. See Remark \ref{sec5.10}(a).

(b) {  Theorem \ref{sec5.11}(ii)} holds for the more general operator $F$ as in Lemma \ref{sec5.10}. $\Box$
\end{rem}

\begin{lem}\label{sec5.120}{(Maximum principle)} {  Let $F$ satisfy (\ref{sec5.8}) and $F(t, r, p, O)=0$, for any $t\ge 0$, any $r\in \IR$ and any $p\in \IR^n$.

(a) Suppose that $u\in usc(lsc)(\Om_T)$ solves
$F(t,u, Du, D^2u)-u_t\ge(\le)0,\;\mbox{in $\Om_T$.}$}
Then $u\le \sup_{P_T}u$ $(u\ge \inf_{P_T} u)$.

{  (b) Let $k>1$. Suppose that, in addition, $F$ satisfies Condition B. Let $f:\IR^+\rightarrow \IR^+$ be a $C^1$ increasing function and $f^{1/(k-1)}$ be concave.
Assume that $u\in usc(lsc)(\Om_T),\;u>0,$ solves}
$$F(t,u, Du, D^2u)+\chi(t)|Du|^k-f(u)u_t\ge(\le)0,\;\;\mbox{in $\Om_T$,}$$
where $\chi$ is a continuous function. Then $u\le \sup_{P_T}u$ $(u\ge \inf_{P_T} u)$. 
\end{lem}
\begin{proof} Since $F(t, r, p, O)=0$, for any $(t,r,p),\;t\ge 0$, the function $\phi=\sup_{P_T} u$ is a solution. Similarly, $\eta=\inf_{P_T}u$ is also a solution. Using Remark \ref{sec3.18}, Theorem \ref{sec5.11} and Remark \ref{sec5.12}, the claims hold.
\end{proof}

\begin{rem}\label{sec5.13} Let $F$ satisfy (\ref{sec5.8}), Condition B and $F(t, r, p, O)=0$. Suppose that $u\in usc(\Om_T\cup P_T)$ solves
$$(*)\qquad F(t,u,Du, D^2u)+\chi(t)|Du|^k-u^{k-1}u_t\ge (\le)0,\;\;\mbox{in $\Om_T$.}$$
If $u>0$ and $\phi=\log u$ then by Lemma \ref{sec3.15}, 
$$F(t,e^\phi, D\phi, D^2\phi+D\phi\otimes D\phi)+\chi |D\phi|^k-\phi_t\ge(\le)0,$$ 
in $\Om_T$. {  Remark \ref{sec5.10}(a) and (\ref{sec5.8})(ii)} show that if $u>0$ is a sub-solution of $(*)$ and $v>0$ is super-solution of $(*)$ then
$$\frac{u}{v}\le \max\left( \sup_{P_T}\frac{u}{v},\;1\right).$$
If $F=F(t,p, X)$ then Remark \ref{sec5.10}(b) shows that $u/v\le \sup_{P_T}(u/v).$
 
A similar quotient type comparison principle was derived for the doubly nonlinear parabolic equations studied in  \cite{BL2, BL3}. $\Box$
\end{rem}

\begin{rem}\label{sec5.14} Let $H$ be as in Theorem \ref{sec5.11}, $k\ge 1$, and 
$f(u)=u^m,\;m\ge 0$.
The condition $f^{1/(k-1)}(\tht),\;k>1,$ is concave in $\tht$ implies that $0\le m\le k-1$ and Theorem \ref{sec5.11} holds. For $k=1$, we require that $m=0$. 
For $m<0$ or $m>k-1$, it is not clear to us if a comparison principle holds. $\Box$
\end{rem}
\vsp

\section{Initial data $t=0$. Constructions for Theorems \ref{I} and \ref{II}.}

In Sections 5, 6 and 7, we address the existence of positive solutions to {  (\ref{sec2.001}), i.e,} 
\eqRef{sec6.1}
H(Du, D^2u)+\chi(t)|Du|^\G-f(u)u_t=0,\;\;\mbox{in $\Om_T$ and $u(x,t)=h(x,t),\;\forall(x,t)\in P_T$,}
\ee
where $h$ is as in (\ref{sec3.30}), $\G>0$ and $f:[c,\infty)\rightarrow [0,\infty),\;c\ge 0,$ is $C^1$. Let us recall that
$$h(x,t)=\left\{\begin{array}{lcr} i(x),  \qquad\;\;\;\forall x\in \Om,\;\mbox{at $t=0$,}\\ j(x,t),\qquad  \forall(x,t)\in \p\Om\times[0,T), \end{array}\right.$$
where $i(x)$ and $j(x,t)$ are positive and continuous and, for any $y\in \p\Om$, 
$\lim_{x\rightarrow y}i(x)=\lim_{(z,t)\rightarrow (y,0)}j(z,t)=j(y,0)$, where
$x\in \Om$ and $z\in \p\Om.$

We assume in Sections 5, 6 and 7 that {  
\bea\label{sec6.100}
&&\mbox{(i) $k>1$ and $f:[0,\infty)\rightarrow [0,\infty)$ is an increasing $C^1$ function and $f^{1/(k-1)}(\tht)$ is } \nonumber\\ 
&&\mbox{concave in $\tht$, and}\;\;\G=k,\;\;\;\;\mbox{or} \nonumber\\
&&\mbox{(ii) $k\ge 1$, $f(\tht)=1,\;\forall \tht\in \IR$, and $0<\G<\g$.} 
\eea }
This will ensure that the problem in (\ref{sec6.1}) has a comparison principle. We also assume through out that
\eqRef{sec6.2}
\mbox{either (i) $\inf_{0\le \tht<\infty} f(\tht)>0$, \quad or \quad (ii) $f\ge 0$ and $f(\tht)=0$ iff $\tht=0$.}
\ee 

Our proof of the existence of a positive continuous solution to the problem (\ref{sec6.1})
involves constructing positive sub-solutions and super-solutions for the problem that are arbitrarily close, in a local sense, to the data specified on the parabolic boundary $P_T$. Existence then follows by using
Perron's method \cite{CIL}, {see also \cite{BL2}.} Uniqueness is implied by Theorem \ref{sec5.11}. The ideas used are an adaptation of the works in \cite{BL2, BL3}. 

We have divided our work into three sections. In this section we take up the construction for the initial data at $t=0$. {  Our work is valid for any bounded domain $\Om$.}

Set $\vartheta=\inf_{P_T} h$ and $M=\sup_{P_T} h$. Assume that 
\eqRef{sec6.3}
0<\vartheta\le M<\infty,\;\;\;\mbox{and}\;\;\;0<\om=\inf_{[\vartheta/2,\;2M]}f(\tht)\le \sup_{[\vartheta/2,\; 2M]}f(\tht)=\nu<\infty.
\ee
If $\vartheta=M$ then $M$ is the solution. Through out the rest of the work, the quantity $\ve>0$ is small and so chosen that
\eqRef{sec6.4}
0<\frac{\vartheta}{2}<\vartheta-2\ve\le M-2\ve\le 2M.
\ee
{  Also, set
\eqRef{sec6.40}
B_0=\sup_{[0,T]}|\chi(t)|.
\ee}
Our constructions will ensure that the sub-solutions and the super-solutions $\eta$ of (\ref{sec6.1}) are bounded below by $\vartheta/2$ and bounded above by $2M$. 
\vsp
We start with the {  initial data $h(x,0)$.} We select points $y\in \overline{\Om}$ at $t=0$. There are two cases to consider: (a) $y\in \Om$, and (b) $y\in \p\Om$. We assume that $h(y,0)>\vartheta$. If $h(y, 0)=\vartheta$, we take the sub-solution to be $\vartheta$. Similarly, if $h(y,0)=M$, we take the super-solution to be $M$.
\vsp
We recall the following calculation. Let $g^\pm(x)=a\pm br^2,\;a,\;b\ge 0$, where $r=|x-z|$ for some $z\in \IR^n$. 
By (\ref{sec2.2}), (\ref{sec2.3}) and Remark \ref{sec3.120} ,
\bea\label{sec6.7}
&&(i)\qquad\left( 2b \right)^k r^{k_1}m(0)\le H(Dg^+, D^2g^+)\le \left( 2b\right)^k r^{k_1} \mu(0),\;\;\;\mbox{and}\nonumber\\
&&(ii)\qquad-\left( 2b\right)^k r^{k_1}\mu(0)\le H(Dg^-, D^2g^-)\le -\left( 2b\right)^k r^{k_1} m(0).
\eea
Recall the definitions of $m(\lam)$ and $\mu(\lam)$, see (\ref{sec2.6}) and (\ref{sec2.5}). Thus, $\mu(0)\ge m(0)>0$. 
\vsp
{\bf Part I's of Theorems \ref{I} and \ref{II}. Case (\ref{sec6.100})(i): $k>1$, {  $f$ increasing $C^1$ function,} $f^{1/(k-1)}$ is concave
 and $\G=k$.} 

{\bf Case (a):} Let $y\in \Om$ and $\ve>0$, small, so that (\ref{sec6.4}) holds. By continuity, there is a $0<\dl_0\le$dist$(y, \p\Om)$ such that 
$$h(y,0)-\ve \le h(x,0)\le h(y,0)+\ve,\;\;\forall x\in B_{\dl_0} (y).$$ 
Recall by the comment right after (\ref{sec6.7}) that $\mu(0)>0$. Set $r=|x-y|$.
\vsp
{\bf Sub-solution:} Note that $k_1>0$, see (\ref{sec2.2}) and (\ref{sec2.3}). Define
\eqRef{sec6.8}
\tau=\frac{1}{\ell}\log\left(\frac{h(y,0)-2\ve}{\vartheta-2\ve}\right),\;\;\;\;b=\frac{1-e^{-\ell \tau}}{\dl^2}\;\;\;\mbox{and}\;\;\;\ell=\frac{3(8b)^k \dl^{k_1} \mu(0) M^{2k-1}}{\om\vartheta^k}
\ee
where $0<\dl\le \dl_0.$ Note that $\ell={  E}\dl^{-k_1-2k_2},$ where { $E$} is independent of $\tau$ and $\dl$. 
First we choose $\dl>0$, small, and calculate $\ell$, $b$ and $\tau$. In particular, choose $\dl$ small so that $\tau<T$.

Using (\ref{sec6.8}), let $R$ be the region
\eqRef{sec6.9}
R=\{(x,t):\;\; e^{\ell(\tau-t)}(1-br^2)\ge1,\;\;0\le t\le \tau.\}
\ee
The base of $R$ is a spatial sphere of radius $\dl$ at $t=0$, tapers as $t$ increases and has an apex at $(y,\tau)$. We construct a bump like function at $(y,0)$ which decreases in $t.$ 

Next, define
\eqRef{sec6.10}
\eta(x,t)=\left\{\begin{array}{lcc} (\vartheta-2\ve)e^{\ell(\tau-t)}(1-br^2), & \forall(x,t)\in R,\\ \vartheta-2\ve, & \forall (x,t)\in (\Om_T\cup P_T)\setminus R. \end{array}\right.
\ee
By (\ref{sec6.8}), (\ref{sec6.9}) and (\ref{sec6.10}), $0\le br^2\le 1-e^{\ell(t-\tau)},\;0\le t\le \tau,$
\bea\label{sec6.11}
&&(i)\;\; {  \eta(y,0)=\sup \eta =h(y,0)-2\ve},\;\;\;(ii)\;\;\;\eta=\vartheta-2\ve,
\;\mbox{in $\p R\cap \Om_T$},\;\;\;(iii)\;\; \eta\le h,\;\mbox{in $P_T$},\nonumber\\
&&\quad\mbox{and}\;\;
(iv)\;\;\frac{\vartheta}{2}\le \eta\le M.
\eea
Recalling (\ref{sec6.3}), {  (\ref{sec6.40}),} (\ref{sec6.7})(ii), (\ref{sec6.8}), (\ref{sec6.10}), setting { $ A_1=(\vartheta-2\ve)e^{\ell(\tau-t)}$} and estimating
{  $B_0\dl^k\le 2\mu(0) \dl^{k_1}$} (take $\dl$ small),
we calculate in $0\le r\le \dl$ and $0\le t\le \tau$,
\bea\label{sec6.110}
H(D\eta, D^2\eta)+\chi(t)|D\eta|^k-f(\eta) \eta_t &\ge& {  A_1}\ell f(\eta)(1-br^2)-{  A_1^k B_0} (2br)^k-{  A_1^k}(2b)^kr^{k_1}\mu(0)\nonumber\\
&\ge& {  A_1^k}\left[ \frac{\ell f(\eta)(1-b\dl^2)}{[ (\vartheta-2\ve)e^{\ell(\tau-t)}]^{k-1}}-(2b)^k\left({  B_0}\dl^k+\dl^{k_1}\mu(0)\right)\right]\nonumber\\
& \ge&  {  A_1^k}\left[ \frac{\ell f(\eta)e^{-\ell \tau}}{[ (\vartheta-2\ve)e^{\ell(\tau-t)}]^{k-1}}-(2b)^k\left( {  B_0}\dl^k+\dl^{k_1}\mu(0)\right)\right]\nonumber\\
&{  \ge}& {  A_1^k} \left( \frac{\ell \om}{(\vartheta-2\ve)^{k-1}e^{k\ell \tau}}-3(2b)^k\dl^{k_1}\mu(0)\right)\nonumber\\
&\ge& {  A_1^k}    \left( \frac{\ell \om}{M^{k-1}(4M/\vartheta)^k}-3(2b)^k\dl^{k_1}\mu(0)\right)=0,
\eea
where we have used (\ref{sec6.4}) and (\ref{sec6.8}) (i.e, $1-b\dl^2=e^{-\ell\tau}\ge {  \vartheta/(4M) }$).  Thus, $\eta$ is a sub-solution in $R$ and Lemma \ref{sec3.21} shows that $\eta$ is a sub-solution in $\Om_T$.

\vsp
{\bf Super-solution:} The work is similar to what we did for the sub-solution. Define
\eqRef{sec6.120}
\tau=\frac{1}{\ell} \log\left( \frac{M+2\ve}{h(y,0)+2\ve}\right),\;\;\;\;b=\frac{e^{\ell \tau}-1}{\dl^2}\;\;\;\mbox{and}\;\;\;\ell=\frac{3(4b)^k \dl^{k_1}M^{k-1} \mu(0)}{\om},
\ee
where $0<\dl\le \dl_0.$ Again, {  $\ell=O(\dl^{-k_1-2k_2})$} implying that $\tau\rightarrow 0$ if $\dl\rightarrow 0$.

Let 
$$R=\{(x,t):\;\;(1+br^2)e^{\ell(t-\tau)}\le 1,\;\;0\le t\le \tau\}.$$
Then, $br^2\le e^{\ell(\tau-t)}-1$ and, at $t=0$, $R$ is a ball of radius $\dl$. As $t$ increases $R$ tapers to $(y,\tau)$. Define
\eqRef{sec6.13}
\phi(x,t)=\left\{\begin{array}{lcc} (M+2\ve)e^{\ell(t-\tau)}(1+br^2), & \forall(x,t)\in R,\\ M+2\ve, & \forall (x,t)\in (\Om_T\cup P_T)\setminus R. \end{array}\right.
\ee
It is clear that
\bea\label{sec6.14}
&&(i)\;\;{ \phi(y,0)=\inf \phi=h(y,0)+2\ve,}\;\;\;(ii)\;\;\;\phi=M+2\ve,\;\mbox{in $\p R\cap \Om_T,$}\;\;\;(iii)\;\;\phi\ge h,\;\mbox{{  in $P_T$},}\nonumber\\
&&\qquad\mbox{and}\;\;(iv) \;\;\;\vartheta\le \phi\le 2M.
\eea
Set {  $A_2=(M+2\ve)e^{\ell(t-\tau)}$}. We calculate in $R$ using (\ref{sec6.3}), 
{  (\ref{sec6.40})}, (\ref{sec6.7})(i), (\ref{sec6.120}), (\ref{sec6.14}) and the comment after (\ref{sec6.8}), and see that, for small $\dl$,
\bea\label{sec6.140}
H(D\phi, D^2\phi)+\chi(t)|D\phi|^k-f(\phi)\phi_t
&\le& {  A_2^k}(2b)^k r^{k_1}\mu(0)+{  A_2^k B_0} (2br)^k-\ell\om {  A_2}(1+br^2)\nonumber\\
&\le& {  A_2^k}\left[ (2b)^k \left( \dl^{k_1}\mu(0)+{  B_0}\dl^k\right)-\frac{\ell\om}{[ (M+2\ve){  e^{\ell(t-\tau)} } ]^{k-1} }\right]\nonumber\\
&\le& {  A_2^k}\left( 3(2b)^k \dl^{k_1}\mu(0)-\frac{\ell\om}{(2M)^{k-1} }\right) {  \le 0.}
\eea
Thus, $\phi$ is a super-solution in $R\cap \Om_T$. Recalling (\ref{sec6.13}) and using 
Lemma \ref{sec3.21}, $\phi$ is a super-solution in $\Om_T$. 
\vsp
{\bf Case (b)} Let $y\in \p\Om:$ By continuity, there are $\dl>0$ and $s>0$ such that 
$$h(y,0)-\ve\le h(x,t)\le h(y,0)+\ve,\;\; {   \forall (x,t)\in P_T\cap (B_{\dl}(y)\times[0, s]). }$$

{  We utilize the quantities in (\ref{sec6.8}) and (\ref{sec6.120}) in our constructions.
For both the sub-solution and the super-solution, we take the $\ell$'s large enough so that $\tau\le s$ and the apex $(y,\tau)\in B_\dl(y)\times [0,s]$. }

Next, we define the sub-solution $\eta$ as in (\ref{sec6.10}) {  and the super-solution $\phi$} as in (\ref{sec6.13}).
The rest of the work is similar to part (a). $\Box$
\vsp
{\bf The Part II's of Theorems \ref{I} and \ref{II} . Case (\ref{sec6.100})(ii): $k\ge 1$, {  $f(\tht)=1, \;\forall \tht\in \IR,$} and any $0<\G<\g$.} 

We consider 
\eqRef{sec6.17}
H(Du, D^2u)+\chi(t)|Du|^\G-u_t=0,\;\;\mbox{in $\Om_T$ and $u=h$ in $P_T$.}
\ee
For both the sub-solution and the super-solution we proceed as in Part I. 

Let $\eta$ be as in (\ref{sec6.10}), and $\phi$ be as in (\ref{sec6.13}).
We discuss the changes needed in (\ref{sec6.8}) and (\ref{sec6.120}).  Note that unlike Part I, $k_1=0$ may occur, i.e, $k=1$. We show that the calculations in the corresponding regions $R$ continue to apply by modifying the quantity $\ell$. The proof for the rest of $\Om_T$ is as in Part I. 

We address the sub-solution $\eta$. Let $(y,0)$ be as in Part I. Setting {  $A_3=(\vartheta-2\ve)e^{\ell(\tau-t)}$} 
in $R$ (see (\ref{sec6.9}) and (\ref{sec6.110})), 
\ben
H(D\eta, D^2\eta)+\chi(t)|D\eta|^\G-\eta_t&\ge& {  A_3}\ell (1-br^2)-{  A_3^\G B_0} (2br)^\G-{  A_3^k}(2b)^kr^{k_1}\mu(0)\\
&=&{  A_3^k}\left( \frac{\ell (1-br^2)}{{  A_3^{k-1}}}-\left( {  A_3^{\G-k}B_0} (2br)^{\G}+(2b)^kr^{k_1}\mu(0)\right) \right)\\
&\ge&{  A_3^k}\left( \frac{\ell (1-b\dl^2)}{{  A_3^{k-1}}}-\left( {  A_3^{\G-k}B_0} (2b\dl)^{\G}+(2b)^k\dl^{k_1}\mu(0)\right) \right).
\een
Since {  $\vartheta/2\le A_3\le 2M$}, using the appropriate estimates for { $A_3$}(depending on whether $\G\ge k$ or $\G<k$) and choosing $\ell$ large, it follows that $\eta$ is a sub-solution in $R$ {  and hence in $\Om_T$.}

We now discuss the super-solution $\phi$. Setting {  $A_4=(M+2\ve)e^{\ell(t-\tau)}$} and
calculating in $R$ (see (\ref{sec6.140})),
\ben
H(D\phi, D^2\phi)+\chi(t)|D\phi|^k-\phi_t &=&{  A_4^k}(2b)^k r^{k_1}\mu(0)+{  A_4^\G B_0} (2br)^\G-\ell {  A_4}(1+br^2)\\
&\le &  {  A_4^k}\left(  (2b)^k \dl^{k_1}\mu(0)+{  A_4^{\G-k}B_0}(2b\dl)^\G - \frac{\ell}{{  A_4^{k-1}}} \right).
\een
Since {  $\vartheta/2\le A_4\le 2M$}, arguing as done above, one can choose $\ell$ large enough so that $\phi$ is a super-solution in $R$ {  and thus in $\Om_T$.}
\vsp
\section{Side Boundary: Case \eqref{sec2.8}(i). Construction for Theorem \ref{I}.}

We construct positive sub-solutions and super-solutions for the side boundary $\p\Om\times(0,T)$ when
Case (i) in (\ref{sec2.8}) holds. Our results hold for any bounded $\Om$. 

As in Section 5, we assume that $f:[0,\infty)\rightarrow [0,\infty)$ and (\ref{sec6.2}) holds. We present the work for Parts I and II of the theorem below. 

We recall (\ref{sec6.1}) for easy reference: 
\begin{equation}\label{sec7.1}
H(Du,D^2u)+\chi(t)|Du|^\G-f(u) u_t=0,\quad \mbox{in $\Om_T$, and $u= h$ in $P_T$.}
\end{equation}
{  Combining the constructions in this section with the set of sub-solutions and super-solutions in Section 5 and applying the Perron method one obtains the existence
of positive solutions of (\ref{sec7.1}) when (\ref{sec2.8})(i) holds.}
Recall the notations and the conditions stated in (\ref{sec6.3}) and (\ref{sec6.4}).

We recall (\ref{sec2.8})(i): there is a
\eqRef{sec7.3}
1<\bar{\lam}<2\;\;\;\mbox{such that}\;\;\mu(\bar{\lam})<0,
\ee
where $\mu(\lam)=\max\{ m_{max}(\lam),\;-\mu_{min}(\lam)\}$, see (\ref{sec2.4}) and (\ref{sec2.6}).

Fix $\ve>0,$ small, and $(y,s)\in P_T$ where $s>0$. 
By continuity, there is a $\dl_0>0$ and $\tau_0>0$, depending on $y$ and $s$, such that
\eqRef{sec7.30}
h(y,s)-\ve\le h(x,t)\le h(y,s)+\ve,\;\;\;\forall(x,t)\in \overline{D}_{\dl_0,2\tau_0}(y,s)\cap P_T.
\ee
\vsp
Recall from (\ref{sec2.2}) and (\ref{sec2.3}) that $k=k_1+k_2$ and $\g=k_1+2k_2.$
{  Set $r=|x-y|$ and} $v^\pm(r)=a\pm br^\beta,$ where $b>0$ and $\beta>0$.
From Remark \ref{sec3.120}, 
\bea\label{sec7.4}
&(i)& \frac{\left( b\beta\right)^k}{ r^{\g-\beta k}}m(2-\beta)\le H(Dv^+, D^2v^+)
\le \frac{\left( b\beta\right)^k}{ r^{\g-\beta k}} \mu(2-\beta), \nonumber\\
&(ii)&-\frac{\left(b\beta\right)^k}{ r^{\g-\beta k}}\mu(2-\beta)\le H(Dv^-, D^2v^-)\le -\frac{\left(b\beta\right)^k}{ r^{\g-\beta k}} m(2-\beta).
\eea
Also, the assumption in (\ref{sec7.3}) shows that if $2-\beta=\bar{\lam}$ then $\beta=2-\bar{\lam}$ and 
\eqRef{sec7.6}
\mu(2-\beta)=\mu(\bar{\lam})<0,\;\;\;0<\beta<1\;\;\;\mbox{and}\;\;\;\g-\beta k>0.
\ee
{  Recall that $0<\vartheta\le h\le M<\infty$,} $\om=\inf_{[\vartheta/2,2M]}f(\tht)$ and $\nu=\sup_{[\vartheta/2, 2M]}f(\tht)$.
\vsp
\NI{\bf Part I: $k>1$, {  $f>0$ is an increasing $C^1$ function,} $f^{1/(k-1)}$ concave and $\G=k$.}

{\bf Sub-solutions:} Our idea is to construct a sub-solution $\eta$ that will be defined in a region $R$ that lies in $\overline{D}_{\dl_0,2\tau_0}(y,s)$ and extended to the rest of $\Om_T$ as a sub-solution. Moreover, $\vartheta/2\le \eta\le M$. 
Choose 
\bea\label{sec7.8}
&&\ell \tau=\log\left(\frac{h(y,s)-2\ve}{\vartheta-2\ve}\right),\;\;\beta=2-\bar{\lam},\;\;\;{  B_0}=\sup_{[0,T]}|\chi(t)|,\;\;\;0<\dl^{k_2}\le \min\left\{1,\;\dl_0^{k_2},\;\frac{|\mu(\bar{\lam})|}{{  2B_0}}\right\},\nonumber\\
&&\qquad b\dl^\beta=1-e^{-\ell \tau}, \;\;\;\mbox{and}\;\;\;
b\ge \left(\frac{ 2\ell \nu }{|\mu(\bar{\lam})|\beta^k(\vartheta/2)^{k-1}}\right)^{1/k}.
\eea
Choose $\ell$ large so that $0<\tau\le \tau_0$. Next, choose $b$ large so that the lower bound holds and $\dl$ satisfies the conditions.

Set $r=|x-y|$. By (\ref{sec7.8}), $1-br^\beta>0,$ in $[0,\;\dl].$ Define, in $[0\;\dl]\times [s-\tau,\;s+\tau]$, 
\bea\label{sec7.9}
\mbox{the region $R$ to be the set:}\;\;\exp(\ell\tau-\ell|s-t|)(1-br^\beta)\ge 1,\;\;\;|s-t|\le \tau,
\eea
In $R$, $br^\beta\le 1-e^{\ell(|t-s|-\tau)}$ and 
thus, $\overline{R}$ lies in the cylinder $\overline{B}_{\dl}(y)\times [s-\tau,s+\tau]$. {  The set $R$ at the level $t=s$ is the spatial ball $B_\dl(y)$} (see (\ref{sec7.8})) and tapers to the points $(y,s\pm \tau)$ as $|s-t|\rightarrow \tau$.
\vsp
In $\overline{\Om}_T$, define the {\it{bump function} } 
\eqRef{sec7.10}
\eta(x,t)=\eta(r,t)=\left\{ \begin{array}{lcr} (\vartheta-2\ve)\exp(\ell\tau-\ell |s-t|)(1-br^\beta), & \forall(x,t)\in R,\\ 
\vartheta-2\ve,& \forall(x,t)\in \overline{\Om}_T\setminus R.
\end{array}\right.
\ee
From (\ref{sec7.30}), (\ref{sec7.6}), (\ref{sec7.8}), (\ref{sec7.9}) and (\ref{sec7.10}) we see that
\ben
&(i)&\max_{\Om_T}\eta=\eta(y,s)=h(y,s)-2\ve,\;\;(ii) \; \eta \ge \vartheta-2\ve, \;\mbox{in $\overline{\Om}_T$},\;\;(iii)\;\eta=\vartheta-2\ve,\;\;\mbox{in $\p R\cap \Om_T,$}\\
&(iv)&\vartheta-2\ve\le\eta\le h(y,s)-2\ve\le h,\;\mbox{in $R\cap P_T,$}\quad \mbox{and}\;\;(v)\;\eta\le h,\;\mbox{in $P_T.$}
\een
If we show that $\eta$ is a sub-solution in $R\cap \Om_T$ then by Lemma \ref{sec3.21} $\eta$ is a sub-solution in $\Om_T$. This together with
the above listed observations in (i)-(iv) would imply that $\eta$ is a sub-solution of (\ref{sec7.1}). 

{  Let $(x,t)\in R\cap \Om_T$. We discuss separately the two cases: (a) $t\ne s$, and (b) $t=s$. Recall that in $0<r<\dl$, $\eta$ is (i) $C^\infty$ in $x$, and (ii) in $t$, for $t\ne s$.}
\vsp
{\bf Case (a) $t\ne s$:} Call ${  A_5}=(\vartheta-2\ve) e^{\ell \tau-\ell|s-t|}$ and write $\eta={  A_5}(1-br^\beta)$. Using (\ref{sec7.4}), (\ref{sec7.6}), (\ref{sec7.8}), (\ref{sec7.10}) 
and $\g-k\beta-k(1-\beta)=\g-k=k_2$, we get 
\bea\label{sec7.100}
H(D\eta,D^2\eta)&+&\chi(t)|D\eta|^k-f(\eta)\eta_t \ge
\frac{-{  A_5}^k( b\beta)^k \mu(2-\beta) }{r^{\g-\beta k} }-{  A_5^kB_0} (b\beta r^{\beta-1})^k-\ell {  A_5} f(\eta)(1-br^{\beta})\nonumber\\
&\ge& {  A_5^k}\left( \frac{( b\beta)^k |\mu(\bar{\lam})|}{r^{\g-\beta k} }-\frac{{  B_0}(b\beta)^k}{r^{k(1-\beta)}} -\frac{\ell f(\eta)}{{  A_5^{k-1}} }\right)=
{  A_5^k}\left[  \frac{(\beta b)^k}{r^{\g-\beta k}} \left( |\mu(\bar{\lam})|-{  B_0} r^{k_2}\right) -\frac{\ell f(\eta)}{{  A_5^{k-1}}}\right] \nonumber\\
&\ge & {  A_5^k}\left( \frac{(b\beta)^k |\mu(\bar{\lam})|}{2\dl^{\g-\beta k}}-\frac{\ell \nu}{(\vartheta/2)^{k-1}}\right)\ge 0,
\eea
where we have used that {  $\vartheta/2\le A_5\le M$, $\g-\beta k>0$ and $\dl^{\g-k\beta}<1$ (since $\dl^{k_2}<1$).}
Hence, $\eta$ is a sub-solution. 
\vsp
{\bf Case (b) $t=s$:} Let $\psi$ be a test function and $(z,s)\in R$ be such that $\eta-\psi$ has a maximum at $(z,s)$.  Then 
for $(x,t)\rightarrow (z,s)$, 
\eqRef{sec7.11}
\eta(x,t) \le  \eta(z,s) +\psi_t(z,s)(t-s)+\langle D\psi(z,s), x-z\rangle+\frac{\langle D^2\psi(z,s)(x-z),x-z\rangle}{2}+o(|t-s|+|x-z|^2).
\ee
{  Since $r>0$, $\eta$ is $C^{\infty}$ in $x$} Using $t=s$ in (\ref{sec7.11}) we get $D\psi(z,s)=D\eta(z,s)$ and $D^2\psi(z,s)\ge D^2\eta(z,s)$. 
Using (\ref{sec7.10}),taking $x=z$ in (\ref{sec7.11}) and $r=|z-y|$. 
$$
\psi_t(z,s)(t-s)+o(|t-s|)\ge (\vartheta-2\ve)(1-br^\beta)\left[\exp(\ell\tau-\ell |t-s|)-\exp(\ell \tau)\right],\;\;\;\mbox{as}\;\;t\rightarrow s.
$$
Hence, 
$$|\psi_t(z,s)|\le \ell (\vartheta-2\ve)e^{\ell \tau}(1-br^\beta).$$ 
Using the observations made above and arguing as in Case (a) (see (\ref{sec7.100})), we get 
\ben
H(D\psi,D^2\psi)(z,s)&+&\chi(s)|D\psi(z,s)|^k-f(\eta(z,s))\psi_t(z,s) \\
&\ge& {  H(D\eta, D^2\eta)(z,s)+\chi(s)|D\eta(z,s)|^k-f(\eta(z,s))\ell (\vartheta-2\ve)e^{\ell \tau}(1-br^\beta)}\\
&\ge& 0.
\een
Thus $\eta$ is a sub-solution in $R\cap \Om_T$. 
\vsp
{\bf Super-solutions:} In this part, we construct a super-solution $\eta$ of (\ref{sec7.1}). Our work is quite similar to the work for the sub-solution. Choose 
\bea\label{sec7.13}
&&\ell \tau=\log\left(\frac{M+2\ve}{h(y,s)+2\ve}\right),\;\;\beta=2-\bar{\lam},\;\;\;{  B_0=\sup_{[0,T]}|\chi(t)|},\;\;\;0<\dl^{k_2}\le \min\left\{1,\;\dl_0^{k_2},\;\frac{|\mu(\bar{\lam})|}{{  2B_0}}\right\},\nonumber\\
&&\qquad b\dl^\beta=e^{\ell \tau}-1,\;\;\mbox{and}\;\; b\ge \left(\frac{8M\ell \nu }{(\vartheta\beta)^k|\mu(\bar{\lam})|} \right)^{1/k}.
\eea
{  We choose $\ell>0$ and $b$ so that $0<\tau\le \tau_0$ and $\dl$ small.}

The region $R$ is defined as follows. 
\ben
R\;\mbox{is the set:} \; (1+br^\beta)\exp( \ell |t-s|-\ell\tau)\le 1,\;\;|s- t|\le \tau.
\een
Clearly, $br^\beta\le e^{\ell (\tau-|s-t|)}-1$, and thus, $\overline{R}\subset \overline{B}_\dl(y)\times[s-\tau, s+\tau]$.  

Define the {\it indent function}  in $\Om_T$ as follows:
\eqRef{sec7.14}
\phi(x,t)= \left\{ \begin{array}{lcr}  (M+2\ve)(1+br^\beta)\exp(\ell |s-t|-\ell\tau), & \forall(x,t)\in R,\\ 
M+2\ve, & \forall(x,t)\in \overline{\Om}_T\setminus R. \end{array} \right.
\ee
Using (\ref{sec7.13}) and (\ref{sec7.14}),
\ben
&(i)&\min_{\Om_T}\phi=\phi(y,s)=h(y,s)+2\ve,\;\;(ii)\; \phi\le M+2\ve, \;\mbox{in $\overline{\Om}_T$},\;\;(iii)
\;\phi=M+2\ve,\;\mbox{in $\p R\cap \Om_T$,}\\
&(iv)&h\le h(y,s)+2\ve\le \phi\le M+2\ve,\;\mbox{in $R\cap P_T$},\;\;\mbox{and}\;\;(iv) \;\;\phi\ge h,\;\mbox{in $P_T.$}
\een

We show that $\phi$ is a super-solution in $R\cap \Om_T$. Lemma \ref{sec3.21} and the observations (i)-(v), listed above, would then imply that $\phi$ is a super-solution of (\ref{sec7.1}). We consider {  the two cases:} (a) $t\ne s$, and (b) $t=s$. 
\vsp
{\bf (a) $t\ne s$:} Noting that $\eta\in C^\infty$, setting ${  A_6}=(M+2\ve)\exp(\ell |s-t|-\ell\tau)$ and applying (\ref{sec7.4})(i) in $0<r\le \dl$, (\ref{sec7.6}), (\ref{sec7.10}) and (\ref{sec7.13}), {  we get}
\bea\label{sec7.140}
H( D\phi,D^2\phi)&+&\chi(t)|D\phi|^k-f(\phi)\phi_t \le \frac{{  A_6^k}\left(b\beta\right)^k \mu(2-\beta)}{r^{\g-\beta k}}+{  A_6^k B_0} (\beta b)^k r^{(\beta-1)k}+\nu {  A_6}\ell (1+br^\beta)\nonumber\\
&\le& {  A_6^k}\left(\frac{\ell \nu (1+b\dl^\beta)}{ {  A_6^{k-1}} }+  {  B_0} (\beta b)^k r^{(\beta-1)k}  
    -\frac{\left(b\beta\right)^k |\mu(\bar{\lam})| }{r^{\g-\beta k}}
\right)\nonumber\\
&= &{  A_6^k}\left[ \frac{\nu \ell e^{\ell \tau}}{ {  A_6^{k-1}} }+ \frac{(b\beta)^k}{r^{\g-k\beta}}\left({  B_0} r^{k_2}-|\mu(\bar{\lam})| \right)
\right]\le {  A_6^k} \left(\frac{4M\ell\nu}{\vartheta^k} -\frac{\left(b\beta\right)^k |\mu(\bar{\lam})| }{2\dl^{\g-\beta k}}
\right)\le 0,
\eea
where we have used that $e^{\ell \tau}\le 4M/\vartheta,$ { $A_6\ge \vartheta,$ $\g-\beta k>0$ and $\dl^{\g-k\beta}<1$ (since $\dl^{k_2}<1$).} Thus $\phi$ is a super-solution.
\vsp
{\bf (b) $t=s$:} Let $\psi$ be a test function and $(z,s)\in R\cap \Om_T$ be such that $\phi-\psi$ has a maximum at $(z,s)$. 
Then, as $(x,t)\rightarrow (z,s)$, 
\eqRef{sec7.15}
\phi(x,t)-\phi(z,s)\ge \psi_t(z,s)(t-s)+\langle D\psi(z,s), x-z\rangle+\frac{ \langle D^2\psi(z,s)(x-z), x-z\rangle}{2}+o(|t-s|+|x-y|^2).
\ee
We take $x=z$ in (\ref{sec7.15}), set $r=|z-y|$ and use (\ref{sec7.14}) to see that
$$\psi_t(z,s)(t-s)+o(|t-s|)\le (M+2\ve)(1+br^\beta)[ \exp(\ell |t-s|-\ell \tau)-\exp(-\ell \tau)]\;\;\mbox{as}\;\;t\rightarrow s.$$
Thus, 
$$|\psi_t(z,s)|\le \ell(M+2\ve)(1+br^\beta)e^{-\ell \tau}.$$ 

{  Since $r>0$, $\phi$ is $C^2$ in $x$. Hence, (\ref{sec7.15}) shows that $D\psi(z,s)=D\phi(z,s)$ and $D^2\psi(z,s)\le D^2\phi(z,s)$. }
Using (\ref{sec7.1}) and arguing as in (a),
\ben
&&H(D\psi,D^2\psi)(z,s)+\chi(s)|D\psi(z,s)|^k-f(\phi) \psi_t(z,s)\\
&&\qquad\le {  H( D\phi,D^2\phi)(z,s)+\chi(s) |D\phi(z,s)|^k+\ell f((\phi(z,s))(M+2\ve)(1+br^\beta)e^{-\ell|s-t|\ell \tau}\le 0.}
\een
Thus, $\phi$ is a super-solution in the interior of $R\cap \Om_T$. 
\vsp
{\bf Part II: $k\ge 1$, {  $f(\tht)=1,\;\forall \tht\in \IR,$} and any $0<\G<\g$.}

As done in Section 5, we provide an outline of the constructions. The value of $b$ in the functions $\eta$ and $\phi$ (see (\ref{sec7.10}) and (\ref{sec7.14})) will undergo a slight change. 
The differential equation reads
$$H(Du, D^2u)+\chi(t)|Du|^\G-u_t=0,\;\;\mbox{in $\Om_T$ and $u=h$ in $P_T$.}$$

Set $a=\g-k\beta-\G(1-\beta)$, where $\beta$ is as in (\ref{sec7.6}). Then 
\eqRef{sec7.17}
a-\beta(\G-k)=\g-k\beta -\G(1-\beta)-\beta(\G-k)=\g-\G.
\ee

We show Case (a) { ($t\ne s$)} for both $\eta$ and $\phi$ in Part I. Case (b) { ($t=s$)} is quite similar to what was done in Part I. 

We start with $\eta$ and use (\ref{sec7.8}), (\ref{sec7.100}) and (\ref{sec7.17}), to get
\ben
&&H(D\eta,D^2\eta)+\chi(t)|D\eta|^\G-\eta_t \ge
\frac{ { -A_5^k} ( b\beta)^k \mu(2-\beta) }{r^{\g-\beta k} }-{  B_0} ( {  A_5} b\beta r^{\beta-1})^\G-\ell {  A_5}  (1-br^{\beta})\\
&&\ge  {  A_5^k} \left( \frac{( b\beta)^k |\mu(\bar{\lam})|}{r^{\g-\beta k} }-\frac{ {  B_0A_5^{\G-k} } (b\beta)^\G}{r^{\G(1-\beta)}} -\frac{\ell }{ {  A_5^{k-1} } }\right)=
{  A_5^k }  \left[  \frac{(\beta b)^k}{r^{\g-\beta k}} \left( |\mu(\bar{\lam})|-( {  A_5} b\beta)^{\G-k} {  B_0}  r^{a}\right) -\frac{\ell }{ {  A_5^{k-1} } }\right]\\
&&\quad \ge {  A_5^k} \left[  \frac{(\beta b)^k}{r^{\g-\beta k}} \left( |\mu(\bar{\lam})|-( {  A_5} \beta)^{\G-k} {  B_0}  {  \dl^{\g-\G} } \right) -\frac{\ell }{  {  A_5^{k-1}  } } \right]\ge
{  A_5^k} \left( \frac{(b\beta)^k |\mu(\bar{\lam})|}{2\dl^{\g-\beta k}}-\frac{\ell}{(\vartheta/2)^{k-1}}\right)\ge 0,
\een
where (in the second term {  $B_0(A_5b\beta)^{\G-k}r^a$}) we have used that {  $b\le \dl^{-\beta}\le r^{-\beta}$, $r^{a-\beta(\G-k)}=r^{\g-\G}$,} $\dl$ is small and $b$ is large enough. This verifies that $\eta$ is a sub-solution.

Next, we use (\ref{sec7.13}), (\ref{sec7.14}), (\ref{sec7.17}), (\ref{sec7.140}), $e^{\ell \tau}\le 2M/\vartheta$ and see that
\ben
&&H( D\phi,D^2\phi)+\chi(t)|D\phi|^\G-\phi_t \le \frac{ {  A_6^k} \left(b\beta\right)^k \mu(2-\beta)}{r^{\g-\beta k}}+{  B_0} ( {  A_6} \beta b)^\G r^{(\beta-1)\G}+ {  A_6} \ell (1+br^\beta)\\
&&\le {  A_6^k} \left(\frac{\ell  (1+b\dl^\beta)}{ {  A_6^{k-1} } }+  {  B_0} {  A_6^{\G-k} } (\beta b)^\G r^{(\beta-1)\G}  
    -\frac{\left(b\beta\right)^k |\mu(\bar{\lam})| }{r^{\g-\beta k}}
\right)\\
&&\le {  A_6^k} \left[ \frac{ \ell e^{\ell \tau}}{ {  A_6^{k-1} } }+ \frac{(b\beta)^k}{r^{\g-k\beta}}\left( {  B_0} (  {  A_6 } b\beta)^{\G-k} r^{a}-|\mu(\bar{\lam})| \right)
\right]\\
&&\le {  A_6^k} \left[ \frac{ \ell e^{\ell \tau}}{  {  A_6^{k-1} } }+ \frac{(b\beta)^k}{r^{\g-k\beta}}\left( \left(\frac{2M  {  A_6} \beta}{\vartheta}\right)^{\G-k}  {  B_0} {  \dl^{\g-\G} }-|\mu(\bar{\lam})| \right) 
\right]\le {  A_6^k} \left(\frac{4M\ell}{\vartheta^k} -\frac{\left(b\beta\right)^k |\mu(\bar{\lam})| }{2\dl^{\g-\beta k}}
\right)\le 0,
\een
by  using in the third line $b=(e^{\ell \tau}-1)\dl^{-\beta}\le (2M/\vartheta)\dl^{-\beta} {  \le (2M/\vartheta)r^{-\beta}}$, {  $r^{a-\beta(\G-k)}=r^{\g-\G}$} taking $\dl$ small enough and then $b$ large enough.

\begin{rem}\label{sec7.18} The discussion above shows the existence of positive solutions of 
$$H(Du, D^2u)+\chi(t)|Du|^\G-u_t=0,\;\;\mbox{in $\Om_T$ and $u=h$ in $P_T$,}$$
where $h>0$. For a general $h$, define $\hat{h}=h+2\vartheta.$ Then $\hat{h}>0$ and the above has a positive solution $\hat{u}$. Thus, $u=\hat{u}-2\vartheta$ solves the required
differential equation. $\Box$.
\end{rem}
\vsp
  
\section{Side Boundary: The case \eqref{sec2.8} (ii). Construction for Theorem \ref{II}.}

In this section we assume that (\ref{sec2.8})(ii) holds, that is,
\eqRef{sec8.1}
\mbox{there is a { smallest }$\bar{\lam}\ge 2$ such that $\mu(\lam)<0,\;\forall \lam>\bar{\lam}.$}
\ee
Also, recall (\ref{sec2.4}), (\ref{sec2.6}) and (\ref{sec2.5}). In addition, we impose that 
$\Om$ satisfy a uniform outer ball condition. More precisely: there is a $\rho_0>0$ such that, for each $y\in \p\Om$, if $0<\rho\le \rho_0$ then there is a $z\in \IR^n\setminus \Om$ such that {   the} ball $B_{\rho}(z)\subset \IR^n\setminus\Om$ and $y\in \p B_{\rho}(z)\cap \p\Om.$ 

Our goal is to construct sub-solutions $\eta$ and super-solutions $\phi$ of 
\bea\label{sec8.2}
H(Du, D^2u)+\chi(t)|Du|^\G-f(u)u_t=0,\; \mbox{in $\Om_T$ and $u=h$, in $P_T.$}
\eea

Let $(y,s)\in P_T$ where $s>0$. There is a $\dl_0>0$ and $\tau_0>0$, small, depending on $y$ and $s$, such that
\eqRef{sec8.3}
h(y,s)-\ve\le h(x,t)\le h(y,s)+\ve,\;\;\;\forall(x,t)\in \overline{D}_{\dl_0,2\tau_0}(y,s)\cap P_T.
\ee
Recall that $\vartheta=\inf_{P_T}h$, $M=\sup_{P_T}h$, and assume that $0<\vartheta\le M<\infty$. Fix $\ve>0$, small, such that $\vartheta-2\ve>0$.

As done in Section 6, we recall Remark \ref{sec3.120}: let $b>0,\;\beta>0$ and $v^\pm(r)=a\pm br^{-\beta}.$ Then
\bea\label{sec8.4}
&(i)& -\frac{( b\beta)^k}{r^{k\beta+\g}} \mu(\beta+2)\le H(Dv^+, D^2v^+)
\le -\frac{( b\beta)^k}{r^{\beta k+\g}} m(\beta+2), \nonumber\\
&(ii)&\frac{(b\beta)^k}{ r^{\beta k+\g}}m(\beta+2)\le H(Dv^-, D^2v^-)\le \frac{(b\beta)^k}{ r^{\beta k+\g}} \mu(\beta+2).
\eea

Recall (\ref{sec8.1}) and set $\lam=\beta+2>\bar{\lam}$. Then $\beta>\bar{\lam}-2$ and 
\eqRef{sec8.5}
\mu(\beta+2)=\mu(\lam)<0,\;\;\;\beta>0\;\;\mbox{and}\;\;\;\beta k+\g>0.
\ee
Recall that $\om=\inf_{[\vartheta/2,\;2M]}f(\tht)$ and $\nu=\sup_{[\vartheta/2, \;2M]}f(\tht)$.
\vsp
{\bf Part I: $k>1$, {  $f$ is a $C^1$ increasing function}, $f^{1/(k-1)}$ is concave and $\G=k$.} 
\vsp
{\bf Sub-solutions:} By our hypothesis, let $z\in \IR^n\setminus \Om$ and $0<\rho$ be such that $B_\rho(z)\subset 
\IR^n\setminus\Om$ and $y\in \p B_\rho(z)\cap \p\Om$. Set $r=|x-z|$; the region $R$ will be in the cylindrical shell $(\overline{B}_{2\rho}(z)\setminus B_\rho(z))\times [s-\tau, s+\tau]$, where $\rho$ and $\tau$ will be determined below. We require that this shell be in $D_{\dl_0, 2\tau_0}(y,s)$ and this is achieved if $4\rho\le \dl_0$. 

Set {  $B_0=\sup_{[0,T]}|\chi(t)|$} and choose
\eqRef{sec8.6}
\ell \tau=\log\left( \frac{h(y,s)-2\ve}{\vartheta-2\ve}\right),\;\;\beta>\bar{\lam}-2\;\;\;\mbox{and}\;\;\;0<\rho\le \min\left\{ \frac{\dl_0}{4},\;\frac{1}{2}\left(\frac{|\mu(\lam)|}{2 {  B_0} }\right)^{1/(\g-k)}\right\},
\ee
where $\lam=\beta+2$. We choose $\ell$, large, so that $0<\tau\le \tau_0$. A value of $\rho$ will be chosen later.

We define the region $R$ as follows:
for $\rho\le r\le 2\rho$ and $|s-t|\le \tau$, let
\bea\label{sec8.7}
R\;\mbox{is the region:}\;\; \exp(\ell \tau- \ell|s-t|)  \left[ 1- \left(\frac{1-e^{-\ell \tau}}{1-2^{-\beta}}\right)\left(1-\frac{\rho^\beta}{r^\beta}\right) \right]\ge 1.
\eea
At $t=s$ the region $R$ is the spatial annulus $\rho\le r\le 2\rho$, {  it} tapers as $|t-s|\rightarrow \tau$ and at $|s-t|=\tau$ we get $r=\rho$.
Also, $\overline{R}\subset (\overline{B}_{2\rho}(z)\setminus B_\rho(z))\times [s-\tau, s+\tau]$.
\vsp
\NI Define the following {\it{bump function} }in $\Om_T$:
\eqRef{sec8.8}
\eta(x,t)=\eta(r,t)=\left\{ \begin{array}{lcr}(\vartheta-2\ve) \exp(\ell \tau-\ell|s-t|) \left[ 1- \left(\frac{1-e^{-\ell \tau}}{1-2^{-\beta}}\right)\left(1-\frac{\rho^\beta}{r^\beta}\right) \right],&&
\mbox{in $R$},\\
\vartheta-2\ve, && \mbox{in $\overline{\Om}_T\setminus R$.}
\end{array}\right.
\ee
{  Note that $\vartheta/2\le \eta\le M$}. Using (\ref{sec8.3}), (\ref{sec8.6}), (\ref{sec8.7}) and (\ref{sec8.8}), we get
\bea\label{sec8.9}
&&(i) \;\; { \eta(y,s)=\sup \eta}=h(y,s)-2\ve,\;\;\;(ii) \;\;  \eta\ge \vartheta-2\ve, \;\mbox{in $\overline{\Om}_T$},\;\;\;(iii)\;\; \eta\le h,\;\mbox{in $P_T$,}  \nonumber\\
&&\mbox{and}\;\;(iv)\; \vartheta-2\ve\le\eta\le h(y,s)-2\ve\le h\le M,\;\;\mbox{in $R\cap P_T.$}
\eea
Clearly, if $\eta$ is a sub-solution in $\Om_T$, the observations (\ref{sec8.9})(i)-(iv), listed above, would then imply that $\eta$ is a sub-solution of (\ref{sec8.2}). We first show that $\eta$ is a sub-solution in $R\cap \Om_T$. We consider: (a) $t\ne s$, and (b) $t=s$. Lemma \ref{sec3.21} then shows $\eta$ is a sub-solution in $\Om_T$.
\vsp
{\bf (a) $t\ne s$:} Set 
{  ${\hat A_0}=(\vartheta-2\ve) e^{\ell\tau-\ell|s-t|}$} and {  ${\hat C_0}=(1-e^{-\ell \tau})(1-2^{-\beta})^{-1}$.}
Note $\eta$ is $C^{\infty}$ (in $x$) in $R\cap \Om_T$ and {  $\eta\le {\hat A_0} $.}
Using (\ref{sec8.4})(i), (\ref{sec8.5}), (\ref{sec8.8}) and bounding the spatial part of $\eta$ from above by $1$, we get in $\rho\le r\le 2\rho$, $0<|s-t|\le \tau$, 

\bea\label{sec8.90}
H(D\eta,D^2\eta)&+&\chi(t)|D\eta|^k-f(\eta)\eta_t\ge \frac{ { ( {\hat A_0}{\hat C_0} )^k} (\beta \rho^\beta)^k|\mu(2+\beta)|}{r^{\beta k+\g}}  -\frac{ {  B_0( {\hat A_0} {\hat C_0}\beta\rho^\beta)^k} }{r^{k(1+\beta)}}-\nu\ell {  {\hat A_0} }\nonumber \\
&\ge& {  ( {\hat A_0}{\hat C_0} )^k} \left( \frac{ (\beta \rho^\beta)^k|\mu(\lam)|}{r^{\beta k+\g}} -\frac{ {  B_0} (\beta\rho^\beta)^k}{r^{k(1+\beta)}}
-\frac{\nu \ell}{ {  {\hat A_0}^{k-1} {\hat C_0}^k} } \right)\nonumber\\
&=&{  ( {\hat A_0} {\hat C_0} )^k} \left[   \frac{(\beta \rho^\beta)^k}{r^{k\beta+\g}}\left( |\mu(\lam)| -{  B_0} r^{\g-k}\right)
-\frac{\nu \ell}{ {  {\hat A_0}^{k-1}{\hat C_0}^k}  } \right]\nonumber\\
&\ge&{  ( {\hat A_0} {\hat C_0})^k} \left[   \frac{(\beta \rho^\beta)^k}{(2\rho)^{k\beta+\g}}\frac{|\mu(\lam)|}{2}
-\frac{\nu \ell}{ {  {\hat A_0}^{k-1} {\hat C_0}^k} } \right]\ge 
{  ( {\hat A_0}{\hat C_0})^k} \left( \frac{ \beta^k|\mu(\lam)|}{ 2^{\beta k+\g+1}\;\rho^\g }  -\frac{2^{k-1}\nu \ell}{{  {\hat C_0}^k} \vartheta^{k-1}}\right) \ge 0,
\eea 
where ${  {\hat A_0} } \ge \vartheta/2$ and $\rho$ is chosen small enough. Thus, $\eta$ is sub-solution in $R\cap \Om_T$. 

Part (b) and the rest of the proof is similar to that in Part I of Section 6.
\vsp
{\bf Super-solutions:} We now construct a super-solution $\phi>0$ to (\ref{sec8.2}). The ideas are similar to those in Part I and we make use of (\ref{sec8.4})(i). The ball
$B_\rho(z)$ is the outer ball at $y\in \p\Om$, see the discussion for sub-solutions.

Take $\lam>\bar{\lam}$. Set
\bea\label{sec8.10}
\beta=\lam-2,\;\;\ell\tau=\log\left(\frac{M+2\ve}{h(y,s)+2\ve}\right),\;\;\mbox{and}\;\;\;0<\rho\le \dl_0/4.
\eea
Select $\ell$, large, so that $0<\tau\le \tau_0$. A more precise (and smaller) value of $\rho$ is chosen later.

Define $r=|x-z|$. Let $R$ be the region in $\rho\le r\le 2\rho$, $|s-t|\le \tau$, defined as follows.
\eqRef{sec8.11}
R\;\mbox{is the region:}\;\;\exp( \ell |s-t|-\ell\tau)\left[1+ \left(\frac{e^{\ell \tau}-1}{1-2^{-\beta}}\right) \left(1-\left(\frac{\rho}{r}\right)^\beta\right)\right]\le 1.
\ee
Note that if $t=s$ then the spatial annulus $\rho\le r\le 2\rho$ is in $R$. The region tapers as $|s-t|\rightarrow \tau$ and at $|s-t|=\tau$ we have $r=\rho.$

In $\overline{\Om}_T$, define the {\it{indent function} } 
\eqRef{sec8.12}
\phi(x,t)=\left\{ \begin{array}{lcr} (M+2\ve)\exp( \ell |s-t|-\ell\tau)\left[1+ \left(\frac{e^{\ell \tau}-1}{1-2^{-\beta}}\right) \left(1-\left(\frac{\rho}{r}\right)^\beta\right)\right] , && \forall(x,t)\in R\\ 
M+2\ve,&& \forall(x,t)\in \overline{\Om}_{T}\setminus R.
\end{array}\right.
\ee
Using (\ref{sec8.3}), (\ref{sec8.10}) and (\ref{sec8.11}) we see that
\ben
&&(i) \; {  \phi(y,s)=\inf \phi }=h(y,s)+2\ve,\;\;\;(ii) \; \phi\le M+2\ve, \;\mbox{in $\overline{\Om}_{T}$},\;\;\;(iii)\;\;\phi\ge h,\;\mbox{in $P_{T}$},\\
&&(iv)\; h\le h(y,s)+2\ve\le \phi\le 2M,\;\;\mbox{in $R\cap P_{T}.$}
\een
{  Note that $\vartheta/2\le \phi\le 2M$.} We now show that {  $\phi$} is a super-solution in $R\cap \Om_{T}$. We consider: (a) $t\ne s$, and (b) $t=s$. Lemma \ref{sec3.21} will then show that $\eta$ is a super-solution in $\Om_{T}$. 
 
{\bf (a) $t\ne s$:} Set { ${\hat A_1}=(M+2\ve)e^{ \ell |s-t|-\ell\tau}$, ${\hat C_1}=(e^{\ell \tau}-1)(1-2^{-\beta})^{-1}$} and {  $B_0=\sup_{[0,T]}|\chi(t)|.$} Using (\ref{sec8.4})(ii), {  (\ref{sec8.10}),} (\ref{sec8.12}) and bounding the spatial part of $\phi$ by $2e^{\ell \tau}$, we get in $\rho\le r\le 2\rho$, $0<|t-s|\le \tau$,
\bea\label{sec8.120}
&&H(D\phi, D^2\phi)+\chi(t)|D\phi|^k-f(\phi)\phi_t\le \frac{ { ({\hat A_1}{\hat C_1})^k} \left( \beta\rho^\beta\right)^k}{r^{\beta k+\g}} \mu(\beta+2)+\frac{ {  ({\hat A_1}{\hat C_1})^kB_0 }(\beta \rho^\beta)^k}{r^{k(1+\beta)}}+2\nu\ell {  {\hat A_1} }  e^{\ell \tau}\nonumber\\
&&\qquad\qquad\qquad\qquad\qquad\qquad\quad\qquad\le {  ( {\hat A_1}{\hat C_1})^k} \left( \frac{2\nu \ell  e^{\ell \tau}}{ {  {\hat A_1}^{k-1}{\hat C_1}^k} }+\frac{ {  B_0} (\beta\rho^\beta)^k}{r^{k(1+\beta)}}- \frac{\left( \beta\rho^\beta\right)^k}{r^{\beta k+\g}} |\mu(\lam)|\right)\\
&&={  ({\hat A_1}{\hat C_1} )^k} \left[ \frac{2\nu \ell  e^{\ell \tau}}{ {  {\hat A_1}^{k-1}{\hat C_1}^k} }+\frac{(\beta\rho^\beta)^k}{r^{k\beta+\g}}( {  B_0} r^{\g-k}-|\mu(\lam)|) \right] \le {  ( {\hat A_1}{\hat C_1})^k}\left( \frac{2\nu \ell  e^{\ell \tau}}{ {  {\hat A_1}^{k-1}{\hat C_1}^k} }-\frac{(\beta\rho^\beta)^k}{r^{k\beta+\g}}\frac{|\mu(\lam)|}{2} \right)\nonumber\\
&&\le {  ( {\hat A_1} {\hat C_1})^k} \left( \frac{2\nu \ell  e^{\ell \tau}}{ {  {\hat C_1}^k} (\vartheta/2)^{k-1}}- \frac{ {  B_0} \left(\beta\rho^\beta\right)^k|\mu(\lam)|}{2(2\rho)^{\beta k+\g}} \right)
={  ({\hat A_1}{\hat C_1})^k} \left( \frac{2\nu \ell  e^{\ell \tau}}{{  {\hat C_1}^k} (\vartheta/2)^{k-1}}- \frac{ {  B_0} |\mu(\lam)|}{2^{\beta k+\g+1}\rho^\g}\right)\le 0,\nonumber
\eea
where we have used {  ${\hat A_1}\ge \vartheta/2$} and $\rho$ is small. We see that $\phi$ is a super-solution.

The proof of Part (b) and the rest of the proof is similar to that in Part I.
\vsp
{\bf Part II: $k\ge 1$, {  $f(\tht)=1,\;\forall \tht\in \IR,$} and any $0<\G<\g$.}

Our discussion is similar to Part II in Section 6. We will verify that $\eta$ and $\phi$ as in (\ref{sec8.8}) and (\ref{sec8.12}), with slight modifications, continue to be sub-solutions and super-solutions. The differential equation reads
$$H(Du, D^2u)+\chi(t)|Du|^\G-u_t=0,\;\;\mbox{in $\Om_T,$ and $u=h$ in $P_T$.}$$

We compute with $\eta$, see (\ref{sec8.8}) and (\ref{sec8.90}). The definitions of {  ${\hat A_0},\;B_0$ and ${\hat C_0}$} continue to be the same.
\ben
H(D\eta,D^2\eta)&+&\chi(t)|D\eta|^\G-\eta_t\ge \frac{ { ( {\hat A_0} {\hat C_0})^k}  (\beta \rho^\beta)^k|\mu(2+\beta)|}{r^{\beta k+\g}}  -\frac{ {  B_0({\hat A_0}{\hat C_0} \beta\rho^\beta)^\G} }{r^{\G(1+\beta)}}-{  \ell {\hat A_0} } \\
&\ge& {  ({\hat A_0}{\hat C_0})^k} \left( \frac{ (\beta \rho^\beta)^k|\mu(\lam)|}{r^{\beta k+\g}} -\frac{ {  B_0({\hat A_0}{\hat C_0})^{\G-k} } (\beta\rho^\beta)^\G}{r^{\G(1+\beta)}}
-\frac{ \ell}{{  {\hat A_0}^{k-1}{\hat C_0}^k} } \right)\\
&=& {  (  {\hat A_0}{\hat C_0})^k } \left[   \frac{(\beta \rho^\beta)^k}{r^{k\beta+\g}}\left( |\mu(\lam)| -\frac{ {  B_0({\hat A_0}{\hat C_0}\beta)^{\G-k} } \rho^{\beta(\G-k)} }{r^{\beta(\G-k)+\G-\g }} \right)
-\frac{ \ell}{ {  {\hat A_0}^{k-1} {\hat C_0}^k } } \right]\\
&=& {  ( {\hat A_0} {\hat C_0} )^k} \left[   \frac{(\beta \rho^\beta)^k}{r^{k\beta+\g}}\left( |\mu(\lam)| -{  B_0({\hat A_0} {\hat C_0}\beta)^{\G-k} } r^{\g-\G} \left(\frac{\rho}{r}\right)^{\beta(\G-k)}  \right)
-\frac{ \ell}{ {  {\hat A_0}^{k-1}{\hat C_0}^k} } \right]\\
&\ge&{  ( {\hat A_0} {\hat C_0} )^k} \left[   \frac{(\beta \rho^\beta)^k}{(2\rho)^{k\beta+\g} }\frac{|\mu(\lam)|}{2}-\frac{ \ell}{ {  {\hat A_0}^{k-1} {\hat C_0}^k} } \right]
\ge 
{  ( {\hat A_0}{\hat C_0})^k} \left( \frac{ \beta^k|\mu(\lam)|}{2^{\beta k+\g+1}\rho^\g}  -\frac{2^{k-1} \ell}{ {  {\hat C_0}^k}\vartheta^{k-1}}\right) \ge 0,
\een
where we have used $\G<\g$, $\rho\le r\le 2\rho$ and $\rho$ is chosen small. The rest is as in Part I.

Next, we calculate using $\phi$, see (\ref{sec8.12}) and (\ref{sec8.120}). The definitions of { ${\hat A_1},$ $B_0$ and ${\hat C_1}$} continue to be the same. In what follows $\rho\le r\le 2\rho$ and $\rho$ is small. 
\ben
H(D\phi, D^2\phi)&+&\chi(t)|D\phi|^\G-\phi_t\le \frac{ {  ( {\hat A_1}{\hat C_1})^k} \left( \beta\rho^\beta\right)^k}{r^{\beta k+\g}} \mu(\beta+2)+\frac{ {  B_0( {\hat A_1}{\hat C_1}\beta \rho^\beta)^\G} }{r^{\G(1+\beta)}} +2\ell {  {\hat A_1} }  e^{\ell \tau}\\
&\le & {  ( {\hat A_1}{\hat C_1} )^k} \left( \frac{2 \ell  e^{\ell \tau}}{ {  {\hat A_1}^{k-1} {\hat C_1}^k} }+\frac{ {  B_0( {\hat A_1}{\hat C_1})^{\G-k} } (\beta\rho^\beta)^\G}{r^{\G(1+\beta)}}- \frac{\left( \beta\rho^\beta\right)^k}{r^{\beta k+\g}} |\mu(\lam)|\right)\\
&=&{  ( {\hat A_1}{\hat C_1} )^k} \left[ \frac{2 \ell  e^{\ell \tau}}{ {  {\hat A_1}^{k-1} {\hat C_1}^k}  }+\frac{(\beta\rho^\beta)^k}{r^{k\beta+\g}}\left( {  B_0( {\hat A_1}{\hat C_1}\beta)^{\G-k} } r^{\g-\G} \left(\frac{\rho}{r}\right)^{\beta(\G-k)}  -|\mu(\lam)|)\right) \right]\\
&\le &{  ( {\hat A_1}{\hat C_1} )^k} \left( \frac{2 \ell  e^{\ell \tau}}{ {  {\hat A_1}^{k-1} {\hat C_1}^k} }-\frac{(\beta\rho^\beta)^k}{r^{k\beta+\g}}\frac{|\mu(\lam)|}{2} \right)
\le {  ( {\hat A_1}{\hat C_1})^k} \left( \frac{2 \ell  e^{\ell \tau}}{ {  {\hat C_1}^k} (\vartheta/2)^{k-1}}- \frac{ {  B_0} \left(\beta\rho^\beta\right)^k|\mu(\lam)|}{2(2\rho)^{\beta k+\g}} \right)\\
&=&
{  ({\hat A_1}{\hat C_1} )^k} \left( \frac{2 \ell  e^{\ell \tau}}{ {  {\hat C_1}^k} (\vartheta/2)^{k-1}}- \frac{ {  B_0} |\mu(\lam)|}{2^{\beta k+\g+1}\rho^\g}\right)\le 0.
\een
The rest of the proof is as in Part I. Now apply Remark \ref{sec7.18} to get the general statement.

\section{Appendix}

We discuss a maximum principle that applies to the case where $f$ is a positive continuous function. No sign conditions are imposed on the sub-solutions and super-solutions. 

Recall Conditions A and B, (\ref{sec2.1})-(\ref{sec2.6}). From (\ref{sec2.4}), we have
\ben
m_{\min}(\lam)=\min_{|e|=1}H(e, I-\lam e\otimes e),\;\;\mu_{\max}=\max_{|e|=1}H(e, \lam e\otimes e-I)
\een
and $m(\lam)=\min(m_{\min}(\lam),\;-\mu_{\max}(\lam)).$ In place of Condition C, we assume that 
\eqRef{sec9.1}
m(0)>0\;\;\mbox{and}\;\;\lim_{\lam\rightarrow -\infty}m(\lam)=\infty.
\ee 
Recall the notation, $\hat{H}(p,X)=-H(p, -X),\;\forall(p,X)\in \IR^n\times S^n$, see Remark \ref{sec2.9}.

\begin{lem}\label{sec9.2}{(Weak Maximum Principle)} Let $\Om\subset \IR^n,\;n\ge 2$, be a bounded domain and $T>0$. Suppose that $H$ satisfies Conditions A, B and { (\ref{sec9.1}).}  Suppose that $\chi:[0,T]\rightarrow \IR$ and $f:\IR\rightarrow [0,\infty)$ and $f\not\equiv 0,$ are continuous functions. 

Let $\G>0$ and $\phi\in usc(lsc)(\Om_T\cup P_T)$ solve 
$$ H(D\phi, D^2\phi)+\chi(t)|D\phi|^\G- f(\phi) \phi_t\ge (\le) 0,\;\;\mbox{in $\Om_T$}.$$
(a) If $\G\ge k$
then $\sup_{\Om_{T}}\phi\le\sup_{P_T} \phi=\sup_{\Om_T\cup P_T}\phi\;(\inf_{\Om_T} \phi\ge \inf_{P_T}\phi=\inf_{\Om_T\cup P_T}\phi).$\\
(b) If $0<\G<k$ and $\inf f>0$ then the conclusion in (a) holds.\\
(c) If $\chi\equiv 0$ then the conclusion in (a) holds even if $\inf f=0$.
\end{lem}
\begin{proof} 
Let $0<\hat{\tau}<\tau<T$, $\Om_{\hat{\tau},\tau}=\Om\times [\hat{\tau}, \;\tau]$ and $P$ the parabolic boundary of $\Om_{\hat{\tau},\tau}$.
Our goal is to prove the weak maximum principle in $\Om_{\hat{\tau},\tau}$ for any $0<\hat{\tau}<\tau<T$ and then extend it to $\Om_T$. Note that $u$ is bounded from above in 
${\overline\Om_{\hat{\tau},\tau}}$ since $u\in usc(\Om_T\cup P_T)$.

Choose $z\in \IR^n\setminus \Om$ and $R>0$ such that $\Om\subset B_{R}(z)\setminus B_{R/2}(z)$. 
Call $r=|x-z|$; clearly, $R/2\le r\le R,\;\forall x\in \Om$.

Set
\bea\label{sec9.3}
&&\vartheta=\sup_{\Om_{\hat{\tau},\tau}} \phi,\;\;\;\ell=\sup_{P}\phi,\;\;\;\dl=\vartheta-\ell,\;\;\;c=\sup_{\overline{\Om}}\phi(x,\tau),\;\;\eta=\max(\dl,\;\;c-\ell)\;\nonumber\\
&&\mbox{and}\;\;\;\nu=\max(c,\; \vartheta,\;\ell).
\eea
We recall from Remark \ref{sec3.120} (ii) and (\ref{sec7.4})(ii) that if $v=a-br^\beta$, where $b>0$ and $\beta>0$, then
\eqRef{sec9.4}
-\frac{\left(b\beta\right)^k}{ r^{\g-\beta k}}\mu(2-\beta)\le H(Dv, D^2v)\le -\frac{\left(b\beta\right)^k}{ r^{\g-\beta k}} m(2-\beta).
\ee

We argue by contradiction and assume that $\dl>0.$ Since { $\Om_{\hat{\tau},\tau}$} is an open set there is a point $(\xi,\tht)\in { \Om_{\hat{\tau},\tau} }$ such that $\phi(\xi,\tht)>\ell+3\dl/4$ and $0<\hat{\tau}<\tht<\tau$. Define
$$g(t)=0,\;\;\forall t\in [\hat{\tau},\;\tht]\;\;\;\mbox{and}\;\;\;g(t)=(t-\tht)^4/(\tau-\tht)^4,\;\;\forall t\in [\tht,\;\tau].$$

Select $0<\ve\le \min(0.5,\;\dl/4).$ For $\beta>0$, set
\ben
\psi(x,t)=\psi(r,t)=\ell+\frac{\ve}{4} +\eta g(t)-\frac{\ve r^\beta}{32R^\beta},\;\;\forall(x,t)\in \overline{\Om}_{\hat{\tau},\tau}.
\een
Thus, $\psi(x,t)\ge \ell+\ve/8,\;\forall(x,t)\in \overline{\Om}_{\hat{\tau},\tau},$ and $\psi(x,\tau)\ge \ell+\eta+\ve/8\ge c+\ve/8,\;\forall x\in \overline{\Om}$.
Moreover,
\eqRef{sec9.400}
\phi(\xi,\tht)-\psi(\xi, \tht)\ge \ell+\frac{3\dl}{4}-\ell-\frac{\ve}{4}=\frac{3\dl}{4}-\frac{\ve}{4}
\ge \frac{\dl}{4}>0.
\ee
Since $\phi-\psi\le 0$ on $\p \overline{\Om}_{\hat{\tau},\tau}$ and $(\phi-\psi)(\xi, \tht)>0$, the function $\phi-\psi$ has a positive maximum at some point {  $(y,s)\in \Om_{\hat{\tau},\tau}$. }

Set { $B_0=\sup_{[0,T]}|\chi(t)|$}, call $\rho=|y-z|$ and use (\ref{sec9.4}) to get
\bea\label{sec9.5}
H\left(D\psi(y,s), D^2\psi(y,s)\right)&+&\chi(s)|D\psi(y,s)|^\G\le -\left(\frac{\ve \beta}{32 R^\beta}\right)^k   \frac{m(2-\beta)}{\rho^{\g-\beta k}}+{  B_0} \left(\frac{\ve \beta}{32R^\beta}\right)^\G \rho^{(\beta-1)\G}\nonumber\\
&=&\rho^{\beta k-\g}\left(\frac{\ve \beta}{32 R^\beta}\right)^k\left[ {  B_0} \left(\frac{\ve \beta}{32R^\beta}\right)^{\G-k} \rho^{\g-\G+\beta(\G-k)}-m(2-\beta)\right]\nonumber\\
&=&\rho^{\beta k-\g}\left(\frac{\ve \beta}{32 R^\beta}\right)^k\left[ {  B_0} \left(\frac{\ve \beta}{32}
\left(\frac{\rho}{R}\right)^\beta\right)^{\G-k} \rho^{\g-\G}-m(2-\beta)\right]
\eea
Call $I$ the right hand side of the third line in (\ref{sec9.5}) and note that $1/2\le \rho/R\le 1$. We now show part (a) of the lemma. {  Note that $\psi_t(y,s)\ge 0$}.

(i) If $\G>k$ then taking $\beta=2${  (see (\ref{sec9.1}))} and $\ve$ small enough we can make $I<0$. We conclude from (\ref{sec9.5}) that $I<0\le f(\phi(y,s))\psi_t(y,s)$ implying that the lemma holds for
$0<\hat{\tau}<\tau<T$

(ii) If $\G=k$ then $\g-k=k_2$ (see (\ref{sec2.2}) and (\ref{sec2.3})). Taking $\beta$ large and using
(\ref{sec9.1}) we can make $I<0$. We conclude from (\ref{sec9.5}) that $I<0\le f(\phi(y,s))\psi_t(y,s)$ implying that the lemma holds for $0<\hat{\tau}<\tau<T.$

Taking { $B_0=0$} and arguing as above we get part (c) of the lemma.

To see part (b), set { $\om=\inf f$} and modify $g(t)=(t/\tau)^\al$, where $\al$ is large so that $\eta g(\tht)\le \ve/8$. Since $\ve \le \dl/4$, this ensures that
in (\ref{sec9.400})
\ben 
\phi(\xi,\tht)-\psi(\xi, \tht)\ge \ell+\frac{3\dl}{4}-(\ell+\frac{\ve}{4}+\eta g(\tht))\ge\frac{3\dl}{4}-\frac{\ve}{4}
-\frac{\ve}{8}\ge \frac{\dl}{4}>0.
\een
Using (\ref{sec9.5}) estimate $I$ (disregard the second term in the parenthesis) as
\ben
I\le A \left(\frac{\ve \beta}{32R^\beta}\right)^\G \rho^{(\beta-1)\G}=\frac{A}{\rho^\G}\left(\frac{\ve \beta\rho^\beta}{32R^\beta}\right)^\G\le 2^\G A \left(\frac{\beta}{32}\right)^\G \frac{\ve^\G}{R^\G}.
\een
Next, $\psi_t(y,s)=\al \eta s^{\al-1}/\tau^\al\ge\al \eta \hat{\tau}^{\al-1}/\tau^\al$ implying that
\ben
I-f(\phi(y,s))\psi_t(y,s)\le  2^\G A\left(\frac{\beta}{32}\right)^\G \frac{\ve^\G}{R^\G}-\al { \om} \eta (\hat{\tau}^{\al-1}/\tau^\al)<0,
\een
if $R$ is chosen large enough. Using (\ref{sec9.5}), we get a contradiction and $\phi\le \ell$ in $\Om_{\hat{\tau}, \tau}.$

If $\sup_{\Om_T}\phi>\sup_{P_T} \phi$ then there is a point $(y,s)\in \Om_T$ (with $0<s<T$) such that 
$\phi(y,s)>\sup_{P_T} \phi$. Select $0<\hat{s}<s<\bar{s}<T$ and call $P$ the parabolic boundary of $\Om_{\hat{s},\bar{s}}$. Then, $\sup_{P_T}\phi<\phi(y,s)\le \sup_{\Om_{\hat{s},\bar{s}}}\phi\le  \sup_{P}\phi\le \sup_{P_T}\phi.$ This is a contradiction and the lemma holds. 

To show the weak minimum principle, take $v=-\phi$ and conclude that $H(-Dv, -D^2v)\le f(-v)(-v_t)$. If $\hat{f}(v)=f(-v)$ then (\ref{sec2.2}) shows that $\hat{H}(Dv, D^2v)\ge \tilde{f}(v)v_t$. As noted in Remark \ref{sec2.9}, 
$\hat{H}$ satisfies Conditions A, B and { (\ref{sec9.1})} and the minimum principle follows.
\end{proof}

{  \begin{rem} Suppose that  $u$ solves $H(Du, D^2u)+\chi(t)|Du|^\G-f(u)u_t=g(x,t),$ where $L=\sup_{\Om_T}|g|<\infty$ and $\om=\inf_{\IR} f>0$. Using $u\pm\ell t$, $\ell\ge L/\om$ large, one gets
$\inf_{P_T}(u+\ell t)-\ell t\le u\le \sup_{P_T}(u-\ell t)+\ell t.$
\end{rem} }

\NI Department of Mathematics\hfil
Western Kentucky University\hfil
Bowling Green, Ky 42101
\vsp
\NI Department of Liberal Arts\hfil
Savannah College of Arts and design\hfil
Savannah, Ga 31401

\end{document}